\documentclass[reqno,english]{amsart}
\usepackage{amsfonts,amsmath,latexsym,verbatim,amscd,mathrsfs,color,array}
\usepackage[colorlinks=true]{hyperref}

\hypersetup{
	colorlinks=true,
	linkcolor=red
}

\allowdisplaybreaks

\usepackage{float}
\usepackage{amsmath,amssymb,amsthm,amsfonts,graphicx,color,mathtools}
\usepackage[hmargin=2.5cm, vmargin=2.5cm]{geometry}
\usepackage{bbm}
\usepackage{amssymb}
\usepackage{pdfsync}
\usepackage{epstopdf}
\usepackage{cite}
\usepackage{graphicx}
\usepackage{multirow}
\usepackage[font=bf,aboveskip=15pt]{caption}
\usepackage[toc,page]{appendix}
\usepackage[colorlinks=true]{hyperref}
\allowdisplaybreaks[4]

\newcommand{\1}{\mathbf{1}}

\newcommand{\pp}{ {\partial} }

\newcommand{\RR}{{{\mathbb R}}}

\newcommand{\R} {\mathbb R}

\newcommand{\cuad}{{\sqcap\kern-.68em\sqcup}}

\newcommand{\be}{\begin{equation}}
\newcommand{\ee}{\end{equation}}

\newtheorem{lemma}{Lemma}[section]
\newtheorem{proposition}{Proposition}[section]
\newtheorem{theorem}{Theorem}[section]
\newtheorem{corollary}{Corollary}[section]
\newtheorem{remark}{Remark}[section]
\newcommand{\bremark}{\begin{remark} \em}
	\newcommand{\eremark}{\end{remark} }

\numberwithin{equation}{section}
\begin{document}
\title[Global solutions for the Fujita equation]{Long-time dynamics for the energy critical heat equation in $\R^5$}

\author[Z. Li]{Zaizheng Li}
\address{\noindent
School of Mathematical Sciences, Hebei Normal University, Shijiazhuang 050024, Hebei, P.R. China}
\email{zaizhengli@hebtu.edu.cn}

\author[J. Wei]{Juncheng Wei}
\address{\noindent
Department of Mathematics,
University of British Columbia, Vancouver, B.C., V6T 1Z2, Canada}
\email{jcwei@math.ubc.ca}

\author[Q. Zhang]{Qidi Zhang}
\address{\noindent
Department of Mathematics,
University of British Columbia, Vancouver, B.C., V6T 1Z2, Canada}
\email{qdz@amss.ac.cn}

\author[Y. Zhou]{Yifu Zhou}
\address{\noindent
School of Mathematics and Statistics, Wuhan University, Wuhan 430072, China}
\email{yifuzhou@whu.edu.cn}

	\begin{abstract}
We investigate the long-time behavior of global solutions to the energy critical heat equation in $\R^5$
\begin{equation*}
	\begin{cases}
		\pp_t u=\Delta u+|u|^{\frac{4}{3}} u  ~&\mbox{ in }~ \RR^5 \times (t_0,\infty),\\
		u(\cdot,t_0)=u_0~&\mbox{ in }~ \RR^5.
	\end{cases}
\end{equation*}
For $t_0$ sufficiently large, we show the existence of positive solutions for a class of initial value $u_0(x)\sim  |x|^{-\gamma}$ as $|x|\rightarrow \infty$ with $\gamma>\frac32$ such that  the global solutions behave asymptotically
\begin{equation*}
	\| u(\cdot,t) \|_{L^\infty (\R^5)} \sim
	\begin{cases}
		t^{-\frac{3(2-\gamma)}{2}}  ~&\mbox{ if }~ \frac32<\gamma<2
		\\
		(\ln t)^{-3} ~&\mbox{ if }~ \gamma=2
		\\
		1 ~&\mbox{ if }~ \gamma>2
	\end{cases}
	\mbox{ \ for \ } t >t_0,
\end{equation*}
which is slower than the self-similar time decay $t^{-\frac{3}{4}}$. These rates are inspired by Fila-King \cite[Conjecture 1.1]{FilaKing12}.

%This is motivated by the Fila-King diagram \cite[Conjecture 1.1]{FilaKing12} in the case of $\R^5$.

\end{abstract}

\maketitle

%\tableofcontents

%\begin{abstract}

%\textbf{Keywords:} \\
% \textbf{AMS Subject Classification (2010):}
%\end{abstract}

\bigskip

\section{Introduction and main results}

\medskip

Consider the semilinear heat equation
\begin{equation}\label{heat-eq}
	\begin{cases}
		\pp_t u =\Delta u + |u|^{p-1}u, &
		\mbox{ \ in \ } \RR^n \times (0,\infty) ,
		\\
		u(\cdot,0)=u_0, &
		\mbox{ \ in \ } \RR^n,
	\end{cases}
\end{equation}
with $p>1$. It corresponds to the negative $L^2$-gradient flow of the associated energy functional
$$
\frac12\int_{\R^n} |\nabla u|^2-\frac1{p+1}\int_{\R^n}|u|^{p+1},
$$
which is decreasing along classical solutions.

\medskip

Equation \eqref{heat-eq} has been widely studied since Fujita's celebrated work \cite{Fujita66}. The Fujita equation looks rather simple, but extensively rich and sophisticated phenomena arise, and these are intimately related to the power nonlinearity in a rather precise manner. For instance, the Fujita exponent $p_F$, the Sobolev exponent $p_S$ defined respectively as
\begin{equation*}
	p_F:=1+\frac2n,\qquad
	p_S=\begin{cases}
		\frac{n+2}{n-2}~&\mbox{ for }~n\geq 3\\
		\infty ~&\mbox{ for }~n=1,2
	\end{cases}
\end{equation*}
play an important role in \eqref{heat-eq} concerning singularity formation, long-time dynamics, and many others, and they have been studied intensively in innumerable literature. It is well known that \eqref{heat-eq} possesses a global nontrivial solution $u\ge 0$ if and only if $p>p_F$. Whether or not the steady states exist greatly affects the dynamical behavior of \eqref{heat-eq}. The stationary equation of \eqref{heat-eq} does not have positive classical solutions if and only if $p<p_S$ (see \cite{Gidas-Spruck81} and \cite{Chen-Li91} for instance).  For $p=p_S$, up to translations and dilations, the positive steady state to the Yamabe problem is the well known Aubin-Talenti bubble
\begin{equation*}
	U(x)=\alpha_n (1+|x|^2)^{-\frac{n-2}{2}},
	\quad \alpha_n=\left[n(n-2)\right]^{\frac{n-2}{4}} .
\end{equation*}
Such a profile is commonly used when investigating the mechanism of singularity formation for \eqref{heat-eq} with critical exponent $p=p_S$. On the other hand, Liouville type theorems for the heat flow \eqref{heat-eq} and their applications have also been thoroughly investigated. In the subcritical case $p<p_S$, Pol\'{a}\v{c}ik and Quittner \cite{Polavcik-Quittner06} proved the nonexistence of positive radially symmetric bounded entire solution, and they showed that the global nonnegative radial solution of \eqref{heat-eq} decays to $0$ uniformly as $t\rightarrow \infty$. Pol\'{a}\v{c}ik, Quittner and Souplet \cite{Polacik07Indiana} developed a general scheme connecting parabolic Liouville type theorems and universal estimates of solutions. Recently in \cite{Quittner21}, Quittner proved the optimal Liouville theorems without extra symmetry nor decay assumptions on the solutions for $1<p<p_S$ and showed that the nonnegative global solution of \eqref{heat-eq} must decay to $0$ as $t\rightarrow \infty$.

\medskip

This paper aims to understand possible long-time dynamics for global solutions of \eqref{heat-eq} with $p=p_S$ in $\R^5$. Here we call a solution global if its maximal existence time is infinity. The long-time behavior for the solution of \eqref{heat-eq} is partially motivated by the study of threshold solutions. For any nonnegative, smooth function $\phi(x)$ with $\phi \not\equiv 0$, let us define
\begin{equation*}
	\alpha^*=\alpha^*(\phi) :=\sup\{ \alpha>0: \ T_{\max} (\alpha\phi) =\infty \},
\end{equation*}
and $u^*:=u(x,t;\alpha^* \phi)$ is called the  threshold solution associated with $\phi$.  Roughly speaking, the threshold solution lies on the borderline between global solutions and those that blow up in finite time since for $\alpha\gg \alpha^*$, the nonlinearity dominates the Laplacian and vice versa. At the threshold level, the dynamics for $u^*$ in the pointwise sense might be global and bounded, global and unbounded, or blow up in finite time. Any of these might happen depending on the power nonlinearity and the domain. We refer the readers to
Ni-Sacks-Tavantzis \cite{Ni84JDE}, Lee-Ni \cite{LeeNi}, Galaktionov-V\'azquez \cite{Vazquez97CPAM}, Pol\'{a}\v{c}ik \cite{Polacik2011ARMA}, Quittner \cite{Quittner2017}, and the monograph by Quittner and Souplet \cite{Souplet07book} and their references for comprehensive studies and descriptions of threshold solutions.
On the other hand, the global decaying threshold and non-threshold solutions of Fujita equation have been studied extensively, see \cite{Kavian87AIHP,LeeNi,Kawanago96,Suzuki1999,Gui01JDE,Polacik03MA,Polacik07Indiana,Quittner08,Fila08JMAA,Fila08MA,Suzuki08Indiana} and the references therein.

\medskip

In \cite{Kawanago96}, Kawanago gave a complete description of the asymptotic behavior of the positive solution in the case  $p_F <p <p_S$. Specially, $\| u(\cdot,t;\alpha^* \phi)\|_{L^\infty} \sim t^{-\frac{1}{p-1}}$ for $t>1$. The spatial decay of initial value plays an important role in the long-time behavior of solutions and threshold solutions of \eqref{heat-eq}. For $p\ge p_S$, under the assumption that the initial value $u_0$ is radial, positive, continuous, and $$\lim\limits_{|x|\rightarrow \infty} u_0(x)|x|^{\frac{2}{p-1}} =0,$$
Quittner \cite[Theorem 1.2]{Quittner08} showed that there are no global positive radial solutions with self-similar time decay $t^{-\frac{1}{p-1}}$.
From this point, for $p=p_S$, Fila and King \cite{FilaKing12} predicted formally, via matched asymptotics, the possible decaying/growing rate (in time) of threshold solutions to \eqref{heat-eq} with the radial initial value $u_0$ satisfying
\begin{equation}\label{initial-data}
	\lim\limits_{r\rightarrow \infty} r^{\gamma} u_0(r) = A \mbox{ \ for some \ } A>0 ~\mbox{ and }~ \gamma>\frac{n-2}{2}.
\end{equation}
They conjectured that the threshold solution $u$ of \eqref{heat-eq} with initial value $u_0$ should satisfy
\begin{equation*}
	\lim\limits_{t\rightarrow \infty} \frac{\| u (\cdot,t) \|_{L^\infty(\R^n)}}{ \varphi(t;n,\gamma) } = C
\end{equation*}
for some positive constant $C$ depending on $n$ and $u_0$,
where $\varphi(t;n,\gamma)$ is given as:
\begin{table}[H]
	\begin{center}
		\begin{tabular}{ | c | c | c | c | }
			\hline
			& $\frac{n-2}{2} < \gamma <2$ & $ \gamma=2$ & $ \gamma>2$ \\   \hline
			$n=3$ & $t^{\frac{ \gamma-1}{2}}$ & $t^{\frac{1}{2}} (\ln t)^{-1}$ & $t^{\frac{1}{2}}$
			\\    \hline
			$n=4$ & $t^{-\frac{2- \gamma}{2} } \ln t$ & $1$ & $\ln t$
			\\   \hline
			$n=5$ & $t^{-\frac{3(2- \gamma)}{2}}$ & $(\ln t)^{-3}$ & $1$ \\
			\hline
			$n\geq 6$ & $1$ & $1$ & $1$ \\
			\hline
		\end{tabular}
	\end{center}
	\caption{Fila-King \cite[Conjecture 1.1]{FilaKing12}}\label{FKD}
\end{table}

The case $\gamma>1,~ n=3$ was answered affirmatively by del Pino, Musso and Wei \cite{173D}, where the infinite time blow-up solutions were constructed by the gluing method. The infinite time blow-up solutions are also called grow-up/growing solutions in some literature. The case $\gamma>2,~ n=4$ was solved in \cite{infi4d} recently. Due to the intimate connection with the critical Fujita equation in $\R^4$, the trichotomy dynamics of the 1-equivariant harmonic map heat flow was studied in \cite{tri}.
See also Galaktionov-King \cite{King03JDE}, Cort\'azar-del Pino-Musso \cite{Green16JEMS}, del Pino-Musso-Wei-Zheng \cite{del2018sign} (sign-changing solutions), and Ageno-del Pino \cite{ageno2023infinite} for their counterparts in the case of the bounded domain, where the Dirichlet boundary plays a significant role in determining the blow-up dynamics.

\medskip

This paper addresses the case for $n=5$ in Table \ref{FKD}. We first introduce some notations that we will use throughout the paper.

\medskip

\noindent \textbf{Notations:}
\begin{itemize}
	
	\item We write $a\lesssim b$ (respectively $a \gtrsim b$) if there exists a  constant $C > 0$ independent of $t_0$ such that $a \le  Cb$ (respectively $a \ge  Cb$). Set $a \sim b$ if $b \lesssim a \lesssim b$. Denote $f_1=O(f_2)$ if $|f_1| \lesssim f_2$.
	
	\item For any $x\in \mathbb{R}^n$ with $|x|=\big( \sum\limits_{i=1}^n x_i^2 \big)^{1/2}$,  the Japanese bracket denotes $\langle x\rangle = \sqrt{|x|^2+1}$.
	
	\item
	For any $c \in \R$, we use the notation $c-$ (respectively $c+$) to denote a constant less (respectively greater) than $c$ and can be chosen arbitrarily close to $c$.
	
	\item  $\eta(x)$ is a smooth cut-off function satisfying  $\eta(x)=1$ for $|x|\le 1$, $\eta(x)=0$ for $|x|\ge 2$, and $0\le \eta(x) \le 1$ for all $x\in \mathbb{R}^n$.

\end{itemize}

The main theorem is stated below.
\begin{theorem}\label{5d-main-th}
	Consider
	\begin{equation}\label{u-eq-5d}
		\pp_t u=\Delta u+ |u|^{\frac{4}{3}} u
		\mbox{ \ in \ }
		\mathbb{R}^5 \times (t_0,\infty) .
	\end{equation}
	Given constants $\gamma>\frac{3}{2}$ and $D_0,~D_1$ satisfying $0<D_0\le D_1 <2D_0$, for $t_0$ sufficiently large,
	then there exists a positive solution $u$ of the form
	\begin{equation}\label{u-behavior}
		u =
		15^{\frac{3}{4}}
		\mu^{-\frac{3}{2}}
		\left( 1+
		\left|  \frac{x-\xi}{\mu} \right|\right)^{-\frac{3}{2}}
		\eta\left( \frac{x-\xi}{\sqrt{t}}\right)
		+
		O\left( t^{-\frac{\tilde{\gamma}}{2}} R^5 \ln^2 R \right),
		\quad
		R = \ln \ln t  ,
	\end{equation}
	where $\tilde{\gamma} = \min\left\{ \gamma,~ 3- \right\} $, $\mu=\mu(t),~\xi=\xi(t) \in C^1[t_0,\infty)$ satisfy
	\begin{equation}%\label{mu0-est}
		\mu \sim
		\begin{cases}
			t^{2- \gamma},& \gamma<2\\
			\ln^2 t , &\gamma=2\\
			1, &\gamma>2
		\end{cases},
		\qquad
		|\xi| \lesssim
		R^{-\frac{7}{4}}
		\begin{cases}
			t^{2- \gamma},& \gamma<2\\
			\ln^2 t , &\gamma=2\\
			1, &\gamma>2 .
		\end{cases}
	\end{equation}
	In particular, $\| u(\cdot,t)\|_{L^\infty(\mathbb{R}^5)}= 15^{\frac{3}{4}}
	\mu^{-\frac{3}{2}} \left(1+ O\left(t^{ \max\left\{ 3-2\gamma ,- \frac{\tilde{\gamma}}{2}\right\} } \ln^4 t \right)  \right)$. Moreover, the initial value satisfies
	\begin{equation*}
		u(x,t_0) =  \left(4\pi t_0\right)^{-\frac 52} \int_{\RR^5}
		e^{-\frac{|x-z|^2}{4t_0 } }
		\psi_0(z) dz
		\quad
		\mbox{ \ for \ } |x| > 4t_0^{\frac{1}{2}}
	\end{equation*}
	with an arbitrary function $\psi_0(x)$ satisfying $D_0 \langle x \rangle^{-\gamma} \le \psi_0(x) \le  D_1 \langle x \rangle^{-\gamma} $. Furthermore, if $D_0=D_1$, we have
	\begin{equation}\label{u-limit}
		\lim\limits_{|x|\rightarrow\infty}
		\langle x \rangle^{\gamma} u(x,t_0) = D_0 .
	\end{equation}
	
\end{theorem}

\begin{remark}
	\noindent
	\begin{itemize}
		\item
		The restriction $D_1 <2D_0$ is due to a technical reason in the derivation process of \eqref{F-int-upp} for the case $\gamma\le 2$.
		
		\item
		The scaling rate/dynamics $\mu$ is derived by balancing the heat flow of the initial value and the Aubin-Talenti bubble via the orthogonal condition \eqref{mu0-eq-original}.
		
		\item Consider $\pp_t u=\pp_{rr}u + \frac{n-1}{r} \pp_r u + u^{\frac{n+2}{n-2}}$, $r>0, ~t>0$ with $n\in (4,6)$. It is possible to deduce similar results by redoing the construction process.
		
		\item The scaling rate with logarithmic correction $t^{k_1}  (\ln t)^{k_2} (\ln\ln t)^{k_3}\cdots$ for some $k_i \in \RR$, $i\in \mathbb Z_+$ with finite multiplicity can be expected when we take the initial value of the form $u_{0}(x) \sim \langle x\rangle^{\gamma_1} \langle\ln\langle x\rangle \rangle^{\gamma_2}
		\langle \ln \langle \ln\langle x\rangle \rangle \rangle^{\gamma_3}\cdots$ for some $\gamma_i\in \RR$, $i\in \mathbb Z_+$.
		
	\end{itemize}
\end{remark}

For $p>p_F$, Lee and Ni \cite[Theorem 3.8]{LeeNi} gave positive global solutions of \eqref{heat-eq} with the decay rate
$$\| u(\cdot,t) \|_{L^\infty(\R^n)} \sim t^{-k} ~\mbox{ for any }~ k\in\left[\frac1{p-1},~\frac{n}{2}\right] .
$$
In particular, for $n=5$ and $p=p_S$, $k\in\left[\frac34, \frac52\right]$.

Theorem \ref{5d-main-th} implies a direct consequence that somewhat expands the picture of global dynamics of positive solutions in the critical case $p=p_S$ in $\R^5$ with algebraic decay rate:
\begin{corollary}
	For $n=5,~ p=\frac{7}{3}$, for all $k\in [0,\frac52]$, there exists a global positive solution of \eqref{heat-eq} with the rate
	$\| u(\cdot,t) \|_{L^\infty(\R^5)} \sim t^{-k} $ as $t\to \infty$.
\end{corollary}

The construction of Theorem \ref{5d-main-th} is done by the gluing method recently developed in \cite{Green16JEMS,17HMF}. It is a rather versatile and systematic tool that can be used to investigate the singularity formation for various evolution PDEs, and we refer to \cite{Green16JEMS,17HMF,173D,17halfHMF,18Euler,Eulerfrog} and the references therein.

\medskip

The rest of this paper is devoted to the proof of Theorem \ref{5d-main-th}.

\medskip

\section{Approximate solutions and the gluing system}

Consider the critical heat equation
\begin{equation}\label{u-eq}
		\pp_t u=\Delta u+\left|u\right|^{\frac{4}{n-2}}u
		\mbox{ \ in \ }
		\mathbb{R}^n \times (t_0,\infty) .
\end{equation}
The unique positive solution (up to translations and dilations) of the stationary equation $\Delta u+u^{\frac{n+2}{n-2}} =0$,  is given by the Aubin-Talenti solution
\begin{equation*}
U(x)=\alpha_n (1+|x|^2)^{-\frac{n-2}{2}},
\quad \alpha_n=\left[n(n-2)\right]^{\frac{n-2}{4}}.
\end{equation*}
The corresponding linearized operator $\Delta+\frac{n+2}{n-2}U^{\frac{4}{n-2}}$ has bounded kernels
\begin{equation*}
 Z_i(x)= \partial_{x_i} U(x),
 \quad
   i=1,\cdots, n, \quad		Z_{n+1}(x)=\frac{n-2}{2}U(x)+
   x\cdot \nabla U(x).
\end{equation*}\par
The leading term of the solution to \eqref{u-eq} is taken as the following form
\begin{equation*}
	u_1(x,t)= \mu^{-\frac{n-2}{2}} U \left(y \right)
	\eta\left(\tilde{y} \right)
	+\Psi_0(x,t),
	\mbox{ \ where \ }
y:=\frac{x-\xi}{\mu},
\quad \tilde{y} := \frac{x-\xi}{\sqrt{t}} ,
\end{equation*}
$\mu=\mu(t)>0,~ \xi=\xi(t)\in C^1[t_0,\infty)$ will be determined later, and
\begin{equation*}
	\Psi_0(x,t) =  \left(4\pi t\right)^{-\frac n2} \int_{\RR^n}
	e^{-\frac{|x-z|^2}{4t} }
	\psi_0(z) dz,
\end{equation*}
where $D_0 \langle x \rangle^{-\gamma} \le \psi_0(x) \le  D_1 \langle x \rangle^{-\gamma} $ with some constants $0<D_0\le D_1$. Obviously, $\Psi_0>0$ and
\begin{equation*}
\partial_t \Psi_0 = \Delta \Psi_0,
\quad
\Psi_0(\cdot,0) = \psi_0 .
\end{equation*}
We first give a lemma concerning a precise estimate related to $\Psi_0$.
\begin{lemma}\label{Cauchy-est-x=0}
Given $n>0$, $\gamma\in \mathbb{R}$, $t\ge 1$, then
	\begin{equation}
		\left(4\pi t\right)^{-\frac n2} \int_{\RR^n}
		e^{-\frac{|y|^2}{4t} }
		\langle y \rangle^{-\gamma} dy
		=
		v_{n,\gamma}(t) (C_{n,\gamma} + g_{n,\gamma}(t)),
	\end{equation}
	where
	\begin{equation}\label{vn-gamma-def}
		v_{n,\gamma}(t)
		=
		\begin{cases}
			t^{-\frac{\gamma}{2}},
			&
			\gamma <n
			\\
			t^{-\frac n2} \ln (1+t),
			&
			\gamma =n
			\\
			t^{-\frac n2},
			&
			\gamma >n,
		\end{cases}
	\end{equation}
	\begin{equation}\label{Cn-gn-def}
		C_{n,\gamma} =
		\begin{cases}
			(4\pi)^{-\frac n2} \int_{\RR^n}
			e^{-\frac{|z|^2}{4 } }
			|z|^{- \gamma } d z,
			&
			\gamma<n
			\\
			\left(4\pi \right)^{-\frac n2} \frac{1}{2} \left|S^{n-1}\right| ,
			&
			\gamma=n
			\\
			\left(4\pi \right)^{-\frac n2} \int_{\RR^n}
			\langle y \rangle^{-\gamma} dy ,
			&
			\gamma>n
		\end{cases},
		\qquad
		g_{n,\gamma}(t)
		=
		O\Bigg(
		\begin{cases}
			t^{-1},   & \gamma <n-2
			\\
			t^{-1} \langle \ln t \rangle,  &
			\gamma =n-2
			\\
			t^{ \frac{\gamma -n }{2}},  &
			n-2 < \gamma <n
			\\
			\left( \ln (1+t)  \right)^{-1},
			&
			\gamma = n
			\\
			t^{\frac{n-\gamma}{2}},
			&
			n<\gamma<n+2
			\\
			t^{-1}  \langle \ln t \rangle,
			&
			\gamma=n+2
			\\
			t^{-1} ,
			&
			\gamma>n+2
		\end{cases}
		\Bigg).
	\end{equation}
	
\end{lemma}
The proof of Lemma \ref{Cauchy-est-x=0} is postponed to Appendix \ref{Cauchy-est-App}.

\medskip

Hereafter, we always assume  $t_0\ge 1$ is sufficiently large and $t\ge t_0$.
By Lemma \ref{Cauchy-est-x=0}, we have
\begin{equation}\label{Psi0-0-est}
D_0 v_{n,\gamma}(t) \left(C_{n,\gamma} + g_{n,\gamma}(t) \right) \le \Psi_0(0,t) \le D_1 v_{n,\gamma}(t) \left(C_{n,\gamma} + g_{n,\gamma}(t) \right)  .
\end{equation}
By similar calculation, we have
\begin{equation}\label{nab-Psi0-upp}
	\|\nabla \Psi_0(\cdot, t)\|_{L^{\infty}(\mathbb{R}^n)} \lesssim t^{-\frac{1}{2}}
	v_{n,\gamma}(t) .
\end{equation}
By \cite[Lemma A.3]{infi4d},
\begin{equation}\label{Psi0-rough-upp}
	\Psi_0(x,t) \lesssim  	t^{-\frac{\tilde{\gamma}}{2}}	
	\1_{ |x|\le t^{\frac 12} }
	+
	| x|^{- \tilde{\gamma} }
	\1_{ |x| > t^{\frac 12} } ,
\end{equation}
where $\tilde{\gamma}$ is defined as
\begin{equation}\label{til-gamma}
	\tilde{\gamma}:= \min\left\{ \gamma, 3- \right\} .
\end{equation}

Define the error of $f$ as
\begin{equation*}
	E[f]:=- \pp_t f+\Delta f+|f|^{\frac{4}{n-2}} f .
\end{equation*}
Straightforward computation implies
\begin{equation*}
E[u_1] =
\mu^{-\frac{n}{2}} \dot{\mu} Z_{n+1}(y) \eta\left(\tilde{y}  \right)
+
\mu^{-\frac{n}{2}} \dot{\xi} \cdot
\left( \nabla U \right)\left( y \right) \eta\left( \tilde{y}  \right)
+
\mathcal{E}_{\eta}
+
|u_1|^{\frac{4}{n-2}} u_1
-
\mu^{-\frac{n+2}{2}}  U\left( y \right)^{\frac{n+2}{n-2}} \eta\left( \tilde{y}\right) ,
\end{equation*}
where
\begin{equation}\label{E-eta-def}
\mathcal{E}_{\eta}:=   \mu^{-\frac{n-2}{2}} U(y)
\left( 2^{-1} t^{-1} \tilde{y}  + t^{-\frac{1}{2}} \dot{\xi}\right)\cdot \left( \nabla \eta \right)\left( \tilde{y}  \right)
+
2\mu^{-\frac{n}{2}} t^{-\frac{1}{2}}
\left( \nabla U \right)\left(y\right) \cdot
\left( \nabla \eta\right)\left( \tilde{y}\right)
+
\mu^{-\frac{n-2}{2}} t^{-1} U\left( y\right) \left( \Delta \eta\right)\left( \tilde{y}\right) .
\end{equation}

\medskip

We look for an exact solution $u$ of \eqref{u-eq} in the form
\begin{equation}\label{u-def}
	u =u_1 + \psi(x,t)+\mu^{-\frac{n-2}{2}}\phi\left(\frac{x-\xi}{\mu},t\right) \eta_{R} ,
\quad
\eta_{R} :=\eta\left(\frac{x-\xi}{\mu R }\right),
\quad
 R=R(t)= \ln \ln t .
\end{equation}
We make the ansatz
\begin{equation}\label{self-similar-ansatz}
2\mu R \le \sqrt{t}/9 .
\end{equation}

Direct calculation deduces that
\begin{equation*}
	\begin{aligned}
E[u] = \ &
\left(
\mu^{-\frac{n}{2}} \dot{\mu} Z_{n+1}(y)
+
\mu^{-\frac{n}{2}} \dot{\xi} \cdot
\left( \nabla U \right)\left( y \right)
\right)
 \eta\left( \tilde{y}  \right)
+
\mathcal{E}_{\eta}
+
\mu^{-\frac{n+2}{2}} U(y)^{\frac{n+2}{n-2}}
\left( \eta(\tilde{y})^{\frac{n+2}{n-2}}
-
\eta\left( \tilde{y}\right)
\right)
\\
&
+
\mathcal{N}\left[\psi,\phi,\mu,\xi\right]
+
\frac{n+2}{n-2} \mu^{-2}
U(y)^{\frac{4}{n-2}} \eta(\tilde{y})^{\frac{4}{n-2}}
\left(\Psi_0 + \psi + \mu^{-\frac{n-2}{2}} \phi(y,t) \eta_R \right)
\\
		&
		-\pp_t \psi +\Delta \psi
		-	\mu^{-\frac{n-2}{2}} \pp_t \phi (y,t) \eta_R
		+
		\mu^{-\frac{n+2}{2}} \Delta_y \phi(y, t) \eta_R
		+  \Lambda_1\left[\phi,\mu,\xi\right]
		+
		\Lambda_2\left[\phi,\mu,\xi\right],
	\end{aligned}
\end{equation*}
where
\begin{equation}\label{Lambda1-phi}
	\begin{aligned}
		\Lambda_1\left[\phi,\mu,\xi\right]:= \ & \mu^{-\frac{n+2}{2}} R^{-2} \phi(y, t)  \left( \Delta
		\eta\right)(\frac{y}{R})
		+
		2 \mu^{-\frac{n+2}{2}}  R^{-1}  \nabla_y \phi(y, t) \cdot \left(\nabla \eta \right)(\frac{y}{R})
		\\
		&
		+
		\mu^{-\frac{n-2}{2}} \phi( y, t)
		\left( \nabla \eta\right)(\frac{y}{R}) \cdot \left( \frac{\dot{\xi}}{\mu R}
		+ \frac{y}{R} \frac{\pp_t(\mu R)}{\mu R } \right)  ,
	\end{aligned}
\end{equation}
\begin{equation}\label{Lambda2-phi}
	\Lambda_2\left[\phi,\mu,\xi\right]:=
	\mu^{-\frac n2} \dot{\mu}  \left(\frac{n-2}{2} \phi(y,t)
	+ y \cdot \nabla_y \phi (y,t) \right) \eta_R
	+ \mu^{-\frac{n}{2}} \dot{\xi} \cdot \nabla_y \phi( y, t)   \eta_R  ,
\end{equation}
\begin{equation}\label{N-def}
	\mathcal{N}\left[\psi,\phi,\mu,\xi\right]:=
|u|^{\frac{4}{n-2}} u
-
\mu^{-\frac{n+2}{2}} U(y)^{\frac{n+2}{n-2}} \eta(\tilde{y})^{\frac{n+2}{n-2}}
-
\frac{n+2}{n-2} \mu^{-2}
U(y)^{\frac{4}{n-2}} \eta(\tilde{y})^{\frac{4}{n-2}}
\left(\Psi_0 + \psi + \mu^{-\frac{n-2}{2}} \phi(y,t) \eta_R \right)   .
\end{equation}

In order to make $E[u]=0$, it suffices to solve the following gluing system.
\\
\textbf{The outer problem:}
\begin{equation}\label{outer-problem}
	\partial_t\psi =\Delta \psi+\mathcal{G}\left[\psi,\phi,\mu,\xi\right]
	\mbox{ \ in \ } \R^n \times (t_0,\infty),
	\quad
	\psi(\cdot,t_0)=0 \mbox{ \ in \ }  \R^n ,
\end{equation}
where
\begin{equation}\label{g}
	\begin{aligned}	  & \mathcal{G}\left[\psi,\phi,\mu,\xi\right] :=   \Lambda_1\left[\phi,\mu,\xi\right]
		+
		\Lambda_2\left[\phi,\mu,\xi\right]
		+
		\left(
		\mu^{-\frac{n}{2}} \dot{\mu} Z_{n+1}(y)
		+
		\mu^{-\frac{n}{2}} \dot{\xi} \cdot
		\left( \nabla U \right)\left( y \right)
		\right) \eta\left(\tilde{y}  \right) \left( 1-\eta_R \right)
		+
		\mathcal{E}_{\eta}
		\\
		&
		+
		\mu^{-\frac{n+2}{2}} U(y)^{\frac{n+2}{n-2}}
		\left( \eta(\tilde{y})^{\frac{n+2}{n-2}}
		-
		\eta\left( \tilde{y}\right)
		\right)
		 +
		\mathcal{N}\left[\psi,\phi,\mu,\xi\right]
		+
		\frac{n+2}{n-2} \mu^{-2}
		U(y)^{\frac{4}{n-2}} \eta(\tilde{y})^{\frac{4}{n-2}}
		\left(\Psi_0 + \psi   \right)
		\left( 1-\eta_R \right)
	 ;
	\end{aligned}
\end{equation}
\textbf{The inner problem:}
\begin{equation}\label{inner-problem} \mu^2	\partial_t\phi =\Delta_y \phi
	+\frac{n+2}{n-2}U(y)^{\frac{4}{n-2}}\phi
	+\mathcal{H}\left[\psi,\mu,\xi\right]
	\quad
	\mbox{ \ for \ } t>t_0, \quad y\in B_{4R(t)},
\end{equation}
where
\begin{equation}\label{H-def}
		\mathcal{H}\left[\psi,\mu,\xi\right] :=
\mu\dot{\mu} Z_{n+1}(y)
+\mu \dot{\xi} \cdot \left(\nabla U\right)(y)
+\frac{n+2}{n-2}\mu^{\frac{n-2}{2}} U(y)^{\frac{4}{n-2}}
\left(
\Psi_0(\mu y+\xi,t)
+
\psi(\mu y+\xi, t)
\right)
		.
\end{equation}

We introduce the new time variable
\begin{equation}\label{tau-def}
	\tau=\tau(t) :=\int_{t_0}^t \mu^{-2}(s) ds + C_\tau t_0 \mu^{-2}(t_0), \quad
	\tau_0 :=\tau(t_0),
\end{equation}
with a sufficiently large constant $C_\tau$ independent of $t_0$. Then
\eqref{inner-problem} can be rewritten as
\begin{equation}\label{inner-tau-eq}
	\partial_{\tau}\phi=\Delta_y \phi(y, t(\tau))
+\frac{n+2}{n-2}U(y)^{\frac{4}{n-2}}\phi(y,t(\tau))
+\mathcal{H}\left[\psi,\mu,\xi\right](y,t(\tau))
\quad
\mbox{ \ for \ }
\tau>\tau_0, \quad  y\in B_{4R(t(\tau))}  .
\end{equation}

\medskip

\section{Formal analysis of $\mu$ and $\phi$}

Hereafter, we take $n=5$. As the leading term of $\mu$, $\mu_0$ is determined by the orthogonal condition
\begin{equation}\label{mu0-eq-original}	\int_{B_{4R}}\left( \mu_0\dot{\mu}_0 Z_{n+1}(y)+\frac{n+2}{n-2}\mu_0^{\frac{n-2}{2}} U(y)^{\frac{4}{n-2}}\Psi_0(0, t) \right) Z_{n+1}(y)dy=0 ,
\end{equation}
which is equivalent to
\begin{equation}\label{mu0-eq}
	\dot{\mu}_0 = A(R) \mu_0^{\frac{n-4}{2}}\Psi_0(0, t)
,
\end{equation}
where
\begin{equation}\label{AR-def}
A(R):= -\frac{n+2}{n-2} \frac{\int_{B_{4R}} U(y)^{\frac{4}{n-2}} Z_{n+1}(y) dy}{\int_{B_{4R}}  Z^2_{n+1}(y) dy }
=
\frac{n-2}{2}
\frac{\int_{\mathbb{R}^n} U(y)^{\frac{n+2}{n-2}} dy}{\int_{\mathbb{R}^n} Z_{n+1}^2(y) dy}
\left(
1+
O\left( R^{\max\left\{ -2,4-n\right\}}\right)
\right)
\sim 1
\end{equation}
for $t\ge M$ with $M$ sufficiently large, and here we have used $$\int_{\mathbb{R}^n} U(y)^{\frac{4}{n-2}} Z_{n+1}(y) dy = -\frac{(n-2)^2}{2(n+2)} \int_{\mathbb{R}^n} U(y)^{\frac{n+2}{n-2}} dy.$$ We take a solution of \eqref{mu0-eq} as
\begin{equation}\label{mu-0-explicit}
 \mu_0(t)=
 \left( \frac{6-n}{2} \int_{M}^{t} A(R(s))\Psi_0(0, s) d s
 \right)^{\frac{2}{6-n} }
 .
\end{equation}
By \eqref{Psi0-0-est}, for $t_0\ge 9M$ sufficiently large,
\begin{equation}\label{mu0-est}
	0<\mu_0(t)\sim \mu_{0*}(t):=
	\begin{cases}
		t^{\frac{2- \gamma}{6-n}},& \gamma<2\\
		( \ln t)^{\frac{2}{6-n}}, &\gamma=2\\
			1, &\gamma>2
	\end{cases},
\quad
	\dot{\mu}_0(t)
	\sim
\mu_{0*}(t)^{\frac{n-4}{2}} v_{n,\gamma}(t)  .
\end{equation}
We make the following ansatz about $\mu$:
\begin{equation}\label{mu1-ansatz}
\mu=\mu_0+\mu_1,
\mbox{ \ where \ }
\mu_1=\mu_1(t) \in C^1[t_0,\infty),
	\quad
|\mu_1 |\le \mu_0/9 ,
\quad
 |\dot{\mu}_1 |\le \mu_{0*}^{\frac{n-4}{2}} v_{n,\tilde{\gamma} } /9 ,
\end{equation}
which implies $\frac{8}{9} \mu_0 \le \mu \le \frac{10}{9} \mu_0$ and the ansatz \eqref{self-similar-ansatz} for $\gamma >\frac{3}{2}$ and $t_0$ sufficiently large.

Recall \eqref{tau-def}, then   $\tau(t)$ and $t$ have the following relation
\begin{equation}\label{tau-est}
	\tau(t)
\sim
\begin{cases}
t^{\frac{2+2\gamma-n}{6-n} },
&
\frac{n-2}{2} < \gamma <2
\\
t ( \ln t)^{\frac{4}{n-6}}, &\gamma=2
\\
t
, &\gamma>2  .
\end{cases}
\end{equation}

By the ansatz \eqref{mu1-ansatz}, roughly speaking, the upper bound of $\mathcal{H}\left[\psi,\mu,\xi\right]$ is determined by
\begin{equation}\label{leading-H-est}
\left| \mu\dot{\mu} Z_{n+1}(y) \right|
+
\left|
\frac{n+2}{n-2}\mu^{\frac{n-2}{2}} U(y)^{\frac{4}{n-2}}\Psi_0(0,t) \right|
\lesssim
\mu_{0*}^{\frac{n-2}{2}} v_{n, \tilde{\gamma} }
\langle y\rangle^{\max\left\{-4, 2-n\right\}}  .
\end{equation}
By $\gamma>\frac{n-2}{2}$ and \eqref{tau-est}, we have
\begin{equation}\label{til-v-def}
		\mu_{0*}^{\frac{n-2}{2}} v_{n, \tilde{\gamma} }
		=
		\begin{cases}
			t^{\frac{n-2-2\gamma}{6-n}}  ,
			& \gamma<2
			\\
			\left( \ln t\right)^{\frac{n-2}{6-n}} t^{-1} ,
			& \gamma =2
			\\
			t^{-\frac{\tilde{\gamma} }{2}} ,
			&   2<\tilde{\gamma}<3
		\end{cases}
		\sim
		\tilde{v}_{n, \tilde{\gamma}}(\tau(t)) ,
\mbox{ \ where \ }
\tilde{v}_{n,\tilde{\gamma} }(\tau )
:=
\begin{cases}
	\tau^{-1},& \frac{n-2}{2} < \gamma<2\\
	(\tau \ln \tau)^{-1} , &\gamma=2\\
	\tau^{-\frac{\tilde{\gamma} }{2}} , & 2<\tilde{\gamma}<3 .
\end{cases}
\end{equation}

Taking $n=5$, we introduce the norm to measure the right hand side of the inner problem
\begin{equation*}
\|f\|_{*}:=\sup_{\tau>\tau_0, ~y\in B_{2R(t(\tau))}}\left[ \tilde{v}_{5,\tilde{\gamma} }(\tau)\right]^{-1}\langle y\rangle^{3}|f(y,\tau)| .
\end{equation*}

The linearized operator $\Delta + \frac{7}{3}U^{\frac{4}{3}}$
has only one positive eigenvalue $\gamma_0>0$ such that
\begin{equation}\label{def-Z0Z0}
	\Delta Z_0 + \frac{7}{3}U^{\frac{4}{3}} Z_0=\gamma_0 Z_0,
\end{equation}
where the corresponding eigenfunction $Z_0 \in L^{\infty}(\R^5)$ is radially symmetric and has exponential decay at spatial infinity.
The following linear theory of the inner problem in dimension $5$ is given by \cite[Proposition 7.2]{infi4d} and \cite[Proposition 7.1]{Green16JEMS}.
\begin{proposition}\label{phi-upper-bdd}
Consider
\begin{equation}\label{inner-linear-eq}
	\begin{cases}
		\pp_{\tau} f=\Delta  f + \frac{7}{3}U(y)^{\frac{4}{3}}  f +h
		&
		 \text{ \ for \ } \tau>\tau_0, \quad y\in B_{4R(t(\tau) ) },
		\\
	f(y,\tau_0)=e_0 Z_0(y)
	&
	\text{ \ for \ }
	y\in B_{4R(t(\tau_0))},
	\end{cases}
\end{equation}
where $h$ satisfies $\|h\|_{*}<\infty$ and
	\begin{equation}\label{inner-orthogonal-cond}
		\int_{B_{4R(t(\tau))}} h(y,\tau)Z_j(y) dy=0, \quad \forall \tau\in(\tau_0, \infty),
		\quad   j=1,2,\cdots, 6,
	\end{equation}
	then  for $\tau_0$ suffciently large, there exists a solution $\left(f,e_0\right)= \left(\mathcal{T}_{\rm{in}}[h],
	\mathcal{T}_{e_0}[h]\right)$ as a linear mapping about $h$, which satisfies the estimates
	\begin{equation*}
		\langle y\rangle|\nabla f|+	|f| \lesssim \tilde{v}_{5,\tilde{\gamma} }(\tau)R^{5}\ln R \langle y\rangle^{-6} \|h\|_{*},\quad |e_0|\lesssim \tilde{v}_{5,\tilde{\gamma} }(\tau_0)R(\tau_0)\|h\|_{*}.
	\end{equation*}

\end{proposition}

\begin{remark}
By \eqref{tau-def}, $4 R(t(\tau))$ given here behaves like $\ln \ln \tau $ but does not satisfy the assumption for $R(\tau)$ in \cite[p.37]{infi4d} accurately. In fact, one can repeat the proof of \cite[Proposition 7.2]{infi4d} and \cite[Proposition 7.1]{Green16JEMS} to obtain Proposition \ref{phi-upper-bdd}.
\end{remark}

By Proposition \ref{phi-upper-bdd} and
the convenience for applying the Schauder fixed-point theorem for the inner problem \eqref{inner-tau-eq}, we define the norm
\begin{equation}\label{inner-norm}
	\|g\|_{\rm{in}}:=\sup_{\tau>\tau_0, y\in B_{2R(t(\tau))}}
	\Big[\tilde{v}_{5,\tilde{\gamma} }(\tau)R^{5}(t(\tau)) \ln^2 \left( R( t(\tau) ) \right) \Big]^{-1}
	\langle y\rangle^{6}
	 \big(\langle y\rangle |\nabla g(y,\tau)| + |g(y,\tau)| \big) ,
\end{equation}
and we will solve \eqref{inner-tau-eq} in the space
\begin{equation}\label{Bin-def}
B_{\rm{in}} :=
\left\{ g(x,\tau)  \ | \ g(\cdot,\tau)\in C^1\left( B_{2R(t(\tau))} \right) \mbox{ \ for \ } \tau>\tau_0,
\quad
\|g\|_{\rm{in}} \le 1 \right\} .
\end{equation}

\medskip

\section{Solving the outer problem}

\begin{proposition}\label{outer-exist}
	
Given $\phi \in B_{\rm{in}}$, $\mu_1, ~\xi \in C^1[t_0,\infty)$ satisfying \begin{equation}\label{mu-xi-rou-bound}
	|\mu_1 |\le  \mu_{0*} R^{-\frac12} ,
	\quad |\dot{\mu}_1 |\le  \mu_{0*}^{\frac12} v_{5,\tilde{\gamma}}  R^{-\frac12}  ,
	\quad
	|\xi |\le  \mu_{0*} R^{-\frac12} ,
	\quad
	 |\dot{\xi} |\le  \mu_{0*}^{\frac12}  v_{5,\tilde{\gamma} }  R^{-\frac32} ,
\end{equation}
then for $t_0$ sufficiently large, there exists a unique solution $\psi=\psi[\phi,\mu_1, \xi]$ for the outer problem \eqref{outer-problem} with $n=5$, which satisfies the following estimates:
\begin{align}
&
|\psi|\lesssim 	v_{5,\tilde{\gamma} }  R^{-1}\ln^2 R \left(\1_{|x|\le\sqrt{t}}+ t|x|^{-2}\1_{|x|>\sqrt{t}}\right) ,
\label{psi-est}
\\
&
\| \nabla \psi(\cdot,t)\|_{L^\infty(\mathbb{R}^5)}
\lesssim
v_{5,\tilde{\gamma}}  \mu_{0*}^{-1} R^{-2} \ln^2 R .
\label{nabla-psi-est}
\end{align}

\end{proposition}

\begin{proof}
It suffices to find a fixed point for the following mapping
\begin{equation*}
\psi = \mathcal{T}^{\rm{out}}_5\left[ \mathcal{G}\left[\psi,\phi,\mu,\xi\right] \right],
\end{equation*}
where $ \mathcal{G}\left[\psi,\phi,\mu,\xi\right]$ is given in \eqref{g}, and
\begin{equation*}
\mathcal{T}^{\rm{out}}_5\left[f\right]
:=
\int_{t_0}^t \int_{\R^5}
\left[ 4\pi(t-s)\right]^{-\frac{5}{2}}
e^{ -\frac{|x-z|^2}{4(t-s)}  } f(z,s) dz ds .
\end{equation*}
In this proof, we always assume $t_0$ is sufficiently large and $\int_{t_2}^{t_1} \cdots ds =0$ if $t_1\le t_2$. Obviously, \eqref{mu-xi-rou-bound} implies the ansatz \eqref{mu1-ansatz} as well as \eqref{self-similar-ansatz}. Combining these with \eqref{mu0-est}, we see that there exists a constant $C_{\mu}>9$ sufficiently large such that
\begin{equation}
9 C_{\mu}^{-1} \mu_{0*} <\mu<C_{\mu} \mu_{0*}/9 .
\end{equation}

In what follows, \cite[Lemma A.1, Lemma A.2]{infi4d}
 will be used repetitively to estimate $\mathcal{T}^{\rm{out}}_5[\cdot]$.

Recall $\Lambda_1\left[\phi,\mu,\xi\right]$ given in \eqref{Lambda1-phi}. Using \eqref{self-similar-ansatz}, \eqref{mu-xi-rou-bound}, and  $\gamma>\frac{3}{2}$, we have
\begin{equation*}
	\left| \frac{\dot{\xi}}{\mu R} \right|
	+
	\left| \frac{\pp_t(\mu R)}{\mu R}\right|
	=
	\left| \frac{\dot{\xi}}{\mu R} \right|
	+
	\left| \frac{\dot{\mu}}{\mu} + \frac{\dot{R}}{R}\right| \lesssim \mu^{-2} R^{-2} .
\end{equation*}
For $\phi \in B_{\rm{in}}$, by \eqref{til-v-def},
\begin{equation*}
\langle y\rangle |\nabla_y \phi| + |\phi|
\lesssim   \mu_{0*}^{\frac{3}{2}} v_{5,\tilde{\gamma}} R^5 \ln^2 R
\langle y \rangle^{-6}  .
\end{equation*}
Thus,
\begin{equation}\label{Lambda1-upp}
\left| 	\Lambda_1\left[\phi,\mu,\xi\right] \right|
\lesssim
\mu_{0*}^{-2} R^{-3} \ln^2 R
v_{5,\tilde{\gamma} }  \1_{R\le |y| \le 2R}
\le
\mu_{0*}^{-2} R^{-3} \ln^2 R
v_{5,\tilde{\gamma}} \1_{ |x| \le C_{\mu} \mu_{0*} R}
 .
\end{equation}
Then
\begin{equation*}
\begin{aligned}
	&
\mathcal{T}^{\rm{out}}_5 \left[\mu_{0*}^{-2} R^{-3} \ln^2 R
v_{5,\tilde{\gamma}}  \1_{ |x| \le C_{\mu}\mu_{0*} R} \right]
\lesssim
t^{-\frac{5}{2}} e^{-\frac{|x|^2}{16 t}}
\int_{t_0}^{\frac{t}{2}}
\mu_{0*}^3(s)
\left( R^2 \ln^2 R\right)(s) v_{5,\tilde{\gamma} }(s) ds
\\
& +
\mu_{0*}^{-2} R^{-3} \ln^2 R
v_{5,\tilde{\gamma} }
\left[
\left( \mu_{0*} R \right)^2
\1_{|x| \le \mu_{0*} R }
+
|x|^{-3}
e^{-\frac{|x|^2}{16 t }}
\left( \mu_{0*} R\right)^5
\1_{|x| > \mu_{0*} R }
\right]
\\
\lesssim \ &
w_{o}(x,t):=	v_{5,\tilde{\gamma} }  R^{-1}\ln^2 R  \left(\1_{|x|\le \sqrt{t}}+t|x|^{-2}\1_{|x|>\sqrt{t}}\right) ,
\end{aligned}
\end{equation*}
where we have used the properties $v_{5,\tilde{\gamma} }  R^{-1}\ln^2 R  \gtrsim t^{-\frac{5}{2} +c}$ with a small constant $c>0$ and $\gamma>\frac{3}{2}$ to get the last inequality.  Then
\begin{equation*}
\left|	\mathcal{T}^{\text{out}}_5
	\left[  	\Lambda_1\left[\phi,\mu,\xi\right] \right] \right|
	\le C_o w_o(x,t)/2
\end{equation*}
with a sufficiently large constant $C_o \ge 2$.
For this reason, we define the norm
\begin{equation*}
	\|f\|_{\rm{out}}:=\sup_{t\ge t_0,~x\in \mathbb{R}^5 }\left(w_o(x,t)\right)^{-1}|f(x,t)|,
\end{equation*}
and the outer problem \eqref{outer-problem} will be solved in the space
\begin{equation}\label{Bout-def}
	B_{\rm{out}}:=\{f \ | \  \|f\|_{\rm{out}}\le C_o\}.
\end{equation}

Assume $\epsilon_1>0$ is a sufficiently small constant, which can vary from line to line.
For $\Lambda_2\left[\phi,\mu,\xi\right]$ given in \eqref{Lambda2-phi}, we have
\begin{equation*}
	\left| \Lambda_2\left[\phi,\mu,\xi\right] \right|
	\lesssim
	\mu_{0*}^{-\frac{1}{2}} v_{5,\tilde{\gamma}}^2
	R^5 \ln^2 R
	\langle y\rangle^{-6}
	\1_{|x|\le C_{\mu}\mu_{0*} R }
	\lesssim
	t^{-\epsilon_1}
	\mu_{0*}^{-2} R^{-3} \ln^2 R
	v_{5,\tilde{\gamma}}  \1_{ |x| \le C_{\mu}\mu_{0*} R} ,
\end{equation*}
where we have used $\gamma>\frac{3}{2}$ and the last term has been handled in \eqref{Lambda1-upp}.

By \eqref{self-similar-ansatz}, \eqref{mu-xi-rou-bound}, and $\gamma>\frac{3}{2}$, one has
\begin{equation*}%\label{g}
	\begin{aligned}	
		&
		\left|	\left(
		\mu^{-\frac{5}{2}} \dot{\mu} Z_{6}(y)
		+
		\mu^{-\frac{5}{2}} \dot{\xi} \cdot
		\left( \nabla U \right)\left( y \right)
		\right) \eta\left(\tilde{y}  \right) \left( 1-\eta_R \right)
		\right|
		+
		\left| \mathcal{E}_{\eta} \right|
		+
		\left|
		\mu^{-\frac{7}{2}} U(y)^{\frac{7}{3}}
		\left( \eta(\tilde{y})^{\frac{7}{3}}
		-
		\eta\left( \tilde{y}\right)
		\right)
		\right|
		\\
		&
		+
		\bigg|
		\frac{7}{3} \mu^{-2}
		U(y)^{\frac{4}{3}} \eta(\tilde{y})^{\frac{4}{3}}
		\Psi_0
		\left( 1-\eta_R \right)
		\bigg|
		\\
		\lesssim \ &
		\mu_{0*}^{-2} v_{5,\tilde{\gamma}}
		\langle y \rangle^{-3}
		\1_{ \mu R \le |x-\xi| \le 2\sqrt{t} }
		+
		\mu_{0*}^{-\frac{3}{2}} t^{-1}
		\langle y \rangle^{-3}
		\1_{ \sqrt{t} \le |x-\xi| \le 2\sqrt{t} }
		\\
		\lesssim \ &
		v_{5,\tilde{\gamma}}
		\mu_{0*}
		|x|^{-3}
		\1_{ C_{\mu}^{-1} \mu_{0*} R \le |x | \le 4\sqrt{t} }
		+
		\mu_{0*}^{\frac{3}{2}} t^{-\frac{5}{2}}
		\1_{ \sqrt{t}/2 \le |x | \le 4\sqrt{t} }  ,
	\end{aligned}
\end{equation*}   
and their convolutions can be estimated as
\begin{equation*}
	\begin{aligned}
		&
		\mathcal{T}^{\text{out}}_5 \left[
		v_{5,\tilde{\gamma}}
		\mu_{0*}
		|x|^{-3}
		\1_{ C_{\mu}^{-1} \mu_{0*} R \le |x | \le 4\sqrt{t} }  \right]
		\lesssim
		t^{-\frac{5}{2}} e^{-\frac{|x|^2}{16 t}}
		\int_{t_0}^{\frac{t}{2}}
		v_{5,\tilde{\gamma}}(s)
		\mu_{0*}(s)  s ds
		\\
		& \quad +
		v_{5,\tilde{\gamma}}
		\mu_{0*}
		\left(
		\left( \mu_{0*} R\right)^{-1}
		\1_{|x|\le \mu_{0*} R}
		+
		|x|^{-1}
		\1_{ \mu_{0*} R < |x|\le \sqrt{t}}
		+
		t|x|^{-3} e^{-\frac{|x|^2}{16 t}}
		\1_{|x|>\sqrt{t}}
		\right)
		\lesssim
		\left( \ln R \right)^{-2} \omega_{o}(x,t) ;
\\
&
\mathcal{T}^{\rm{out}}_5 \left[
\mu_{0*}^{\frac{3}{2}} t^{-\frac{5}{2}}
\1_{ \sqrt{t}/2 \le |x | \le 4\sqrt{t} }  \right]
\lesssim
t^{-\frac{5}{2}} e^{-\frac{|x|^2}{16 t}}
\int_{t_0}^{\frac{t}{2}}
\mu_{0*}^{\frac{3}{2}}(s) ds
+
\mu_{0*}^{\frac{3}{2}} t^{-\frac{3}{2}}
\left(  \1_{|x| \le \sqrt{t}}
+
t^{\frac{3}{2}} |x|^{-3} e^{-\frac{|x|^2}{16 t}}
\1_{|x| >\sqrt{t}}
\right)
\lesssim
t^{-\epsilon_1} \omega_{o}(x,t) ,
	\end{aligned}
\end{equation*}
where in the last step, we have used the property $\tilde{\gamma}<3$ in \eqref{til-gamma} and $\gamma>\frac{3}{2}$.

For any $\psi_1,\psi_2\in B_{\rm{out}}$, we have
\begin{equation*}
\begin{aligned}
	&
	\left|
\mu^{-2}
U(y)^{\frac{4}{3}} \eta(\tilde{y})^{\frac{4}{3}}
\left( \psi_1 - \psi_2   \right)
\left( 1-\eta_R \right)
\right|
\lesssim
\mu_{0*}^{-2} \langle y \rangle^{-4}
\| \psi_1-\psi_2\|_{\rm{out}} w_o(x,t) \1_{\mu R \le |x-\xi| \le 2\sqrt{t} }
\\
\lesssim \ &
v_{5,\tilde{\gamma} }  R^{-1}\ln^2 R \mu_{0*}^{2} |x|^{-4}
\1_{C_{\mu}^{-1} \mu_{0*} R \le |x| \le 4\sqrt{t} }
\| \psi_1-\psi_2\|_{\rm{out}} .
\end{aligned}
\end{equation*}
Then
\begin{equation*}
\begin{aligned}
&	
\mathcal{T}^{\rm{out}}_5 \left[
v_{5,\tilde{\gamma} }  R^{-1}\ln^2 R \mu_{0*}^{2} |x|^{-4}
\1_{C_{\mu}^{-1} \mu_{0*} R \le |x| \le 4\sqrt{t} } \right]
\lesssim
t^{-\frac{5}{2}} e^{-\frac{|x|^2}{16 t}} \int_{t_0}^{\frac{t}{2}}
(v_{5,\tilde{\gamma} }  R^{-1}\ln^2 R)(s) \mu_{0*}^2(s) s^{\frac{1}{2}} ds
\\
&
+
v_{5,\tilde{\gamma} }  R^{-1}\ln^2 R \mu_{0*}^2
\left(
\mu_{0*}^{-2} R^{-2} \1_{|x|\le\sqrt{t} }
+
|x|^{-3} e^{-\frac{|x|^2}{16 t}} t^{\frac{1}{2}} \1_{|x|>\sqrt{t}}
\right)
\lesssim 	
R^{-2} w_{o}(x,t) ,
\end{aligned}
\end{equation*}
where we have used $\gamma>\frac{3}{2}$ in the last step.

For $\mathcal{N}\left[\psi,\phi,\mu,\xi\right]$ defined in \eqref{N-def}, given any $\psi \in B_{\rm{out}}$, we estimate
\begin{equation*}
\begin{aligned}
&
\left|  \mathcal{N}\left[\psi,\phi,\mu,\xi\right] \right|
\lesssim
\left(
\left| \mu^{-\frac{3}{2}} U(y) \eta(\tilde{y})\right|^{\frac{1}{3}}
+
\left|
\Psi_0 + \psi + \mu^{-\frac{3}{2}} \phi(y,t) \eta_R
\right|^{\frac{1}{3}}
\right)
\left|
\Psi_0 + \psi + \mu^{-\frac{3}{2}} \phi(y,t) \eta_R
\right|^2
\\
\lesssim \ &
\mu^{-\frac{1}{2}} U(y)^{\frac{1}{3}} \eta(\tilde{y})^{\frac{1}{3}}
\left(
\Psi_0^2 + \psi^2 + \left|\mu^{-\frac{3}{2}} \phi(y,t) \eta_R
\right|^{2}  \right)
+
|\Psi_0|^{\frac{7}{3}}
+
|\psi|^{\frac{7}{3}}
+
\left|\mu^{-\frac{3}{2}} \phi(y,t) \eta_R
\right|^{\frac{7}{3}}
\\
\lesssim \ &
\mu_{0*}^{-\frac{1}{2}} \langle y \rangle^{-1} \1_{|x-\xi|\le 2 t^{\frac{1}{2}}}
\bigg[
t^{-\tilde{\gamma}}	
\1_{ |x|\le t^{\frac 12} }
+
| x|^{- 2\tilde{\gamma} }
\1_{ |x| > t^{\frac 12} }
 +
C_{o}^2
\left(
v_{5,\tilde{\gamma}} R^{-1} \ln^2 R
\right)^2
\left( \1_{|x|\le t^{\frac{1}{2}}}
+
\left(t|x|^{-2} \right)^2 \1_{|x|> t^{\frac{1}{2}} }
\right)
\\
&
  +
\left(
v_{5,\tilde{\gamma}} R^5 \ln^2 R
\langle y\rangle^{-6}
\right)^2
\1_{|x-\xi|\le 2\mu R}
    \bigg]
\\
&
+
t^{-\frac{7}{3} \frac{\tilde{\gamma}}{2}}	
\1_{ |x|\le t^{\frac 12} }
+
| x|^{- \frac{7}{3} \tilde{\gamma} }
\1_{ |x| > t^{\frac 12} }
+
C_{o}^{\frac{7}{3}}
\left(
v_{5,\tilde{\gamma}} R^{-1} \ln^2 R
\right)^{\frac{7}{3}}
\left( \1_{|x|\le t^{\frac{1}{2}} }
+
\left(t|x|^{-2} \right)^ {\frac{7}{3}} \1_{|x|> t^{\frac{1}{2}} }
\right)
\\
&
+
\left(
v_{5,\tilde{\gamma}} R^5 \ln^2 R
\langle y\rangle^{-6}
\right)^{\frac{7}{3}}
\1_{|x-\xi|\le 2\mu R}
\\
\lesssim \ &
\mu_{0*}^{\frac{1}{2}} \left(|x|+\mu_{0*}\right)^{-1}
\1_{|x|\le 4 t^{\frac{1}{2}}}
\left[
C_{o}^2
\left(
v_{5,\tilde{\gamma}} R^5 \ln^2 R
\right)^2
\1_{|x|\le t^{\frac{1}{2}}}
+
| x|^{- 2\tilde{\gamma} }
\1_{ |x| > t^{\frac 12} }
+
C_{o}^2
\left(
v_{5,\tilde{\gamma}} R^{-1} \ln^2 R
\right)^2
\left(t|x|^{-2} \right)^2 \1_{|x|> t^{\frac{1}{2}} }
\right]
\\
&
+
C_{o}^{\frac{7}{3}}
\left(
v_{5,\tilde{\gamma}} R^5 \ln^2 R
\right)^{\frac{7}{3}}
\1_{|x|\le t^{\frac{1}{2}} }
+
| x|^{- \frac{7}{3} \tilde{\gamma} }
\1_{ |x| > t^{\frac 12} }
+
C_{o}^{\frac{7}{3}}
\left(
v_{5,\tilde{\gamma}} R^{-1} \ln^2 R
\right)^{\frac{7}{3}}
\left(t|x|^{-2} \right)^ {\frac{7}{3}} \1_{|x|> t^{\frac{1}{2}} }
\\
\lesssim \ &
C_{o}^2
\left(
v_{5,\tilde{\gamma}} R^5 \ln^2 R
\right)^2
\mu_{0*}^{\frac{1}{2}} \left(|x|+\mu_{0*}\right)^{-1}
\1_{|x|\le t^{\frac{1}{2}}}
+
C_{o}^2
\mu_{0*}^{\frac{1}{2}}
t^{-\frac{1}{2} -\tilde{\gamma}}
\1_{ t^{\frac 12} < |x| \le 4 t^{\frac 12} }
\\
&
+
C_{o}^{\frac{7}{3}}
\left(
v_{5,\tilde{\gamma}} R^5 \ln^2 R
\right)^{\frac{7}{3}}
\1_{|x|\le t^{\frac{1}{2}} }
+
| x|^{- \frac{7}{3} \tilde{\gamma} }
\1_{ |x| > t^{\frac 12} }
+
C_{o}^{\frac{7}{3}}
\left(
v_{5,\tilde{\gamma}} R^{-1} \ln^2 R
\right)^{\frac{7}{3}}
\left(t|x|^{-2} \right)^ {\frac{7}{3}} \1_{|x|> t^{\frac{1}{2}} }  .
\end{aligned}
\end{equation*}
Here, by $\gamma>\frac{3}{2}$, we then estimate their convolutions
\begin{equation*}
\begin{aligned}
	&
\mathcal{T}^{\rm{out}}_5
\left[ \left(
v_{5,\tilde{\gamma}} R^5 \ln^2 R
\right)^2
\mu_{0*}^{\frac{1}{2}} \left(|x|+\mu_{0*}\right)^{-1}
\1_{|x|\le t^{\frac{1}{2}}}
\right]
\lesssim
t^{-\frac{5}{2}} e^{-\frac{|x|^2}{16 t}}
\int_{t_0}^{\frac{t}{2}}
\left( v_{5,\tilde{\gamma}} R^5 \ln^2 R\right)^2(s) \mu_{0*}^{\frac{1}{2}}(s)
s^2 ds
\\
&
\quad
+
\left( v_{5,\tilde{\gamma}} R^5 \ln^2 R\right)^2  \mu_{0*}^{\frac{1}{2}}
\left(
t^{\frac{1}{2}} \1_{|x|\le t^{\frac{1}{2}}}
+
t^2  |x|^{-3} e^{-\frac{|x|^2}{16 t}}
\1_{|x|>t^{\frac{1}{2}}}
\right)
\lesssim t^{-\epsilon_1} w_o(x,t) ;
\\
&
\mathcal{T}^{\rm{out}}_5
\left[ \left(
v_{5,\tilde{\gamma}} R^5 \ln^2 R
\right)^{\frac{7}{3}}
\1_{|x|\le t^{\frac{1}{2}} } \right]
\lesssim
t^{-\frac{5}{2}} e^{-\frac{|x|^2}{16 t}}
\int_{t_0}^{\frac{t}{2}}
\left( v_{5,\tilde{\gamma}} R^5 \ln^2 R\right)^{\frac{7}{3}}(s) s^{\frac{5}{2}} ds
\\
& \quad  +
\left( v_{5,\tilde{\gamma} } R^5 \ln^2 R \right)^{\frac{7}{3}}
\left(
t \1_{|x|\le t^{\frac{1}{2} } }
	+
t^{\frac{5}{2}} |x|^{-3} e^{-\frac{|x|^2}{16 t}}
\1_{|x| > t^{\frac{1}{2} } }
\right)
\lesssim  t^{-\epsilon_1} w_o(x,t) ;
\end{aligned}
\end{equation*}
\begin{equation*}
\mu_{0*}^{\frac{1}{2}}
t^{-\frac{1}{2} -\tilde{\gamma}}
\1_{ t^{\frac 12} < |x| \le 4 t^{\frac 12} }
\lesssim
| x|^{- \frac{7}{3} \tilde{\gamma} }
\1_{ |x| > t^{\frac 12} }  ,
\end{equation*}
where we used the property $\tilde{\gamma}<3$;
\begin{equation*}
\begin{aligned}
	&
\mathcal{T}^{\rm{out}}_5
\left[ | x|^{- \frac{7}{3} \tilde{\gamma} }
\1_{ |x| > t^{\frac 12} } \right]
\lesssim
\bigg[
t^{-\frac{5}{2}}
\int_{t_0}^{\frac{t}{2}}
\begin{cases}
0, & \mbox{ \ if \ } \frac{7}{3} \tilde{\gamma} <5
\\
\langle \ln \left( t s^{-1}\right) \rangle
, & \mbox{ \ if \ } \frac{7}{3} \tilde{\gamma} =5
\\
s^{\frac{5}{2} -\frac{7}{6} \tilde{\gamma}}
, & \mbox{ \ if \ } \frac{7}{3} \tilde{\gamma} >5
\end{cases}
~ds
\quad
+
t^{1-\frac{7}{6} \tilde{\gamma}}
\bigg] \1_{|x|\le t^{\frac{1}{2}}}
\\
&
+
\bigg[
t |x|^{-\frac{7}{3} \tilde{\gamma}}
+
t^{-\frac{5}{2}} e^{-\frac{|x|^2}{16 t}}
\int_{t_0}^{\frac{t}{2}}
\begin{cases}
0, & \mbox{ \ if \ } \frac{7}{3} \tilde{\gamma} <5
\\
\langle \ln ( |x| s^{-\frac{1}{2}}) \rangle
, & \mbox{ \ if \ } \frac{7}{3} \tilde{\gamma} =5
\\
s^{\frac{5}{2} -\frac{7}{6} \tilde{\gamma}}
, & \mbox{ \ if \ } \frac{7}{3} \tilde{\gamma} > 5
\end{cases}
~ds
\bigg] \1_{|x|> t^{\frac{1}{2}}}
\lesssim  t^{-\epsilon_1} w_o(x,t) ;
\end{aligned}
\end{equation*}
\begin{equation*}
\begin{aligned}
&
\mathcal{T}^{\rm{out}}_5\left[
\left(
v_{5,\tilde{\gamma}} R^{-1} \ln^2 R
\right)^{\frac{7}{3}}
\left(t|x|^{-2} \right)^ {\frac{7}{3}} \1_{|x|> t^{\frac{1}{2}} }  \right]
\lesssim
\left[
t^{-\frac{7}{3}} \int_{t_0}^{\frac{t}{2}}
\left( v_{5,\tilde{\gamma}} R^{-1} \ln^2 R \right)^{\frac{7}{3}}(s) s^{\frac{7}{3}} ds
+
t \left( v_{5,\tilde{\gamma}} R^{-1} \ln^2 R \right)^{\frac{7}{3}}
\right]
\1_{|x|\le t^{\frac{1}{2}}}
\\
&
+
|x|^{-\frac{14}{3}}
\left[
t^{\frac{10}{3}}
\left( v_{5,\tilde{\gamma}} R^{-1} \ln^2 R \right)^{\frac{7}{3}}
+
\int_{t_0}^{\frac{t}{2}}
\left( v_{5,\tilde{\gamma}} R^{-1} \ln^2 R \right)^{\frac{7}{3}}(s)
s^{\frac{7}{3}} ds
\right]
\1_{|x| > t^{\frac{1}{2}}}
\lesssim  t^{-\epsilon_1} w_o(x,t) .
\end{aligned}
\end{equation*}

For any $\psi_1,\psi_2\in B_{\rm{out}}$, one has
\begin{equation*}
\begin{aligned}
	&
	\left|
\mathcal{N}\left[\psi_1,\phi,\mu,\xi\right] -
\mathcal{N}\left[\psi_2,\phi,\mu,\xi\right] \right|
\\
= \ &
\left|
\frac{7}{3}
\left( \psi_1-\psi_2\right)
\left[
\left|
\mu^{-\frac{3}{2}} U(y) \eta(\tilde{y}) +\Psi_0 +
\theta \psi_1 + \left( 1-\theta \right)
\psi_2
+
\mu^{-\frac{3}{2}} \phi(y,t) \eta_R
\right|^{\frac{4}{3}}
-
\left| \mu^{-\frac{3}{2}} U(y) \eta(\tilde{y}) \right|^{\frac{4}{3}}
\right]
\right|
\\
\lesssim \ &
\| \psi_1-\psi_2\|_{\rm{out}}
w_{o}(x,t)
\bigg[
\mu^{-\frac{1}{2}} U(y)^{\frac{1}{3}} \eta(\tilde{y})^{\frac{1}{3}}
\left(
|\Psi_0| + C_o w_o(x,t) +
\left|\mu^{-\frac{3}{2}}  \phi(y,t) \eta_R \right|
\right)
\\
&
+
|\Psi_0|^{\frac{4}{3}} + C_o^{\frac{4}{3}} w_o(x,t)^{\frac{4}{3}} +
\left|\mu^{-\frac{3}{2}} \phi(y,t) \eta_R \right|   ^{\frac{4}{3}}
\bigg]
\\
\lesssim \ &
\| \psi_1-\psi_2\|_{\rm{out}}
\bigg[
\mu^{-\frac{1}{2}} U(y)^{\frac{1}{3}} \eta(\tilde{y})^{\frac{1}{3}}
\left(
|\Psi_0|^2 + C_o w_o(x,t)^2 +
\left|\mu^{-\frac{3}{2}}  \phi(y,t) \eta_R \right|^2
\right)
\\
&
+
|\Psi_0|^{\frac{7}{3}} + C_o^{\frac{4}{3}} w_o(x,t)^{\frac{7}{3}} +
\left|\mu^{-\frac{3}{2}} \phi(y,t) \eta_R \right|   ^{\frac{7}{3}}
\bigg] ,
\end{aligned}
\end{equation*}
which can be handled by the same way for estimating $\left|  \mathcal{N}\left[\psi,\phi,\mu,\xi\right] \right|$.

In sum, for $t_0$ sufficiently large,
$\mathcal{T}^{\rm{out}}_5\left[ \mathcal{G}\left[\psi,\phi,\mu,\xi\right] \right] \in B_{\rm{out}}$ and is a contraction mapping about $\psi$, which implies that there exists a unique solution $\psi\in B_{\rm{out}}$. Moreover, by $\gamma>\frac{3}{2}$, we have
\begin{equation*}
\left|  \mathcal{G}\left[\psi,\phi,\mu,\xi\right] \right|
\lesssim
v_{5,\tilde{\gamma}} \mu_{0*}^{-2} R^{-3} \ln^2 R .
\end{equation*}
By the scaling argument, we get \eqref{nabla-psi-est}.

\end{proof}

 \medskip

\section{Solving orthogonal equations about $\mu_1, \xi$}

For the utilization of Proposition \ref{phi-upper-bdd} for the inner problem, we need to choose suitable $\mu_1$, $\xi$ such that the orthogonal conditions
\begin{equation}\label{orth-eq}
	\int_{B_{4R }}\mathcal{H}[\psi, \mu, \xi](y,t) Z_{i}(y) dy=0 ,
	\quad
	\mu=\mu_0+\mu_1,
	\quad i=1, \dots, n+1 , \quad  n=5
\end{equation}
are satisfied, where $\psi=\psi[\phi,\mu_1, \xi]$ is solved by Proposition \ref{outer-exist},
and $\mathcal{H}\left[\psi,\mu,\xi\right]$ is given in \eqref{H-def}.
\begin{proposition}\label{mu1-xi-prop}
	Given $0<D_0\le D_1 <2D_0$,
for $t_0$ sufficiently large,	then there exists a solution $(\mu_1,\xi)= (\mu_1[\phi],\xi[\phi])$ for \eqref{orth-eq} with $n=5$ satisfying
\begin{equation}\label{mu-xi-final}
	\left|\mu_1  \right|
	\lesssim
	\mu_{0*} R^{-\frac{2}{3}},
\quad
|\dot{\mu}_1|
\lesssim
 \mu_{0*}^{\frac{1}{2}} v_{5,\tilde{\gamma}} R^{-\frac{2}{3}},
	\quad
	\left| \xi \right| \lesssim
	\mu_{0*} R^{-\frac{7}{4}} ,
\quad
|\dot{\xi}|
\lesssim
\mu_{0*}^{\frac{1}{2}}  v_{5,\tilde{\gamma}} R^{-\frac{7}{4}}
.
\end{equation}

\end{proposition}

\begin{proof}

First, let us consider the general dimension $n$. By \eqref{H-def}, \eqref{orth-eq} is equivalent to
\begin{align}
&
\dot{\mu}
=
-
\frac{n+2}{n-2}
\left(
\int_{B_{4R}} Z^2_{n+1}(y) dy
\right)^{-1}
\mu^{\frac{n}{2}-2}
\int_{B_{4R}}	\left( \Psi_0(\mu y+\xi,t)
+
\psi(\mu y+\xi, t)
\right) U(y)^{\frac{4}{n-2}} Z_{n+1}(y)    dy ,
\label{mu-eq1}
\\
&
 \dot{\xi} =  \vec{\mathcal{S}}[\mu_1,\xi]
 :=
 \left(\mathcal{S}_1[\mu_1,\xi],\dots,\mathcal{S}_n[\mu_1,\xi] \right)
 ,
 \quad \mbox{ \ for $i=1,2,\dots,n$, \ }
\label{xi-eq1} \\
& \mathcal{S}_i[\mu_1,\xi] :=
-
\frac{n+2}{n-2}
\left( \int_{B_{4R }}
Z_i^2(y) dy \right)^{-1}
\mu^{\frac{n}{2}-2}  \notag  \\
&\qquad\qquad\quad\times
\int_{B_{4R }}
\Big[
\Psi_0(\mu y+\xi,t)
-
\Psi_0(0,t)
+
\psi(\mu y+\xi, t)
-
\psi(0, t)
\Big]
U(y)^{\frac{4}{n-2}} Z_i(y) dy,
\notag
\end{align}
 where we have used the parity of $Z_i(y)$. By $\mu=\mu_0+\mu_1$ and $\mu_0$ satisfying \eqref{mu0-eq},
 we rewrite \eqref{mu-eq1} as
\begin{equation}\label{mu1-eq1}
\dot{\mu}_1
+
\beta(t) \mu_1
=  \mathcal{F} [\mu_1,\xi](t)
,
\end{equation}
where
\begin{equation}
\begin{aligned}
	\beta(t) := \ &
\frac{n+2}{n-2}
\left(
\int_{B_{4R}} Z^2_{n+1}(y) dy
\right)^{-1}
\frac{n-4}{2}
\mu_0^{\frac{n}{2}-3}
\Psi_0(0,t)
\int_{B_{4R}}	  U(y)^{\frac{4}{n-2}} Z_{n+1}(y)    dy
\\
= \ &
\frac{n-4}{n-6}
\frac{A(R)\Psi_0(0,t)}{\int_{M}^{t} A(R(s))\Psi_0(0, s) d s}
\end{aligned}
\end{equation}
with the application of \eqref{mu-0-explicit} in the last step;
\begin{equation}
\begin{aligned}
&
\mathcal{F} [\mu_1,\xi](t):=
-
\frac{n+2}{n-2}
\left(
\int_{B_{4R}} Z^2_{n+1}(y) dy
\right)^{-1}
\bigg[
\mu^{\frac{n}{2}-2}
\int_{B_{4R}}
\psi(\mu y+\xi, t)  U(y)^{\frac{4}{n-2}} Z_{n+1}(y)    dy
\\
&\qquad +
\mu^{\frac{n}{2}-2}
\int_{B_{4R}}
\left(
\Psi_0(\mu y+\xi,t)
-
\Psi_0(0,t)
\right) U(y)^{\frac{4}{n-2}} Z_{n+1}(y)    dy
\\
&\qquad +
\left(
\mu^{\frac{n}{2}-2}
-
\mu_0^{\frac{n}{2}-2}
-
\frac{n-4}{2}
\mu_0^{\frac{n}{2}-3} \mu_1
\right)
\Psi_0(0,t)
\int_{B_{4R}}	  U(y)^{\frac{4}{n-2}} Z_{n+1}(y)    dy
\bigg] .
\end{aligned}
\end{equation}

In order to find a solution $(\mu_1,\xi)$ for the system \eqref{xi-eq1}-\eqref{mu1-eq1}, it suffices to solve the following
fixed point problem about $\dot{\mu}_1, \dot{\xi}$,
\begin{equation}\label{mu1-xi-sys}
\begin{aligned}
	\dot{\mu}_1
= \ &
\mathcal{S}_{n+1}[\mu_1,\xi]
:=
\frac{d}{dt} \left( \int_{\tilde{t}_0}^t \mathcal{F} [\mu_1,\xi](s)  e^{\int_t^s \beta(a)da }
ds \right)
=
-\beta(t)
\int_{\tilde{t}_0}^t \mathcal{F} [\mu_1,\xi](s)  e^{\int_t^s \beta(a)da }
ds
+
\mathcal{F} [\mu_1,\xi](t)
 ,
\\
\mu_1
= \ &
\mu_1[\dot{\mu}_1](t)
:=\int_{\tilde{t}_0}^t \dot{\mu}_1(a) da ,
\quad
\dot{\xi} =  \vec{\mathcal{S}}[\mu_1,\xi],
\quad
\xi
=
\xi[\dot{\xi}](t)
:=\int_{\tilde{t}_0}^t \dot{\xi}(a) da
\mbox{ \ with \ }
 \tilde{t}_0 :=
\begin{cases}
	t_0, &   \gamma\le 2
	\\
	\infty,
	&  \gamma > 2
\end{cases}
\end{aligned}
\end{equation}
if these integrals are well-defined.

Hereafter, we take $n=5$. By \eqref{Psi0-0-est} and \eqref{AR-def},  we have
\begin{equation}\label{beta-sim}
	\left\{
	\begin{aligned}
		&
		-
		\frac{D_1}{D_0}
		\left( 1-
		\frac{\gamma}{2}
		\right) t^{-1}
		\left( 1+ O(R^{-\frac{1}{2}}) \right) \le \beta(t) \le
		-
		\frac{D_0}{D_1} \left( 1-
		\frac{\gamma}{2}
		\right) t^{-1}
		\left( 1+ O(R^{-\frac{1}{2} }) \right) & & \mbox{ \ if \ } \gamma<2 ,
		\\
		&
		-
		\frac{D_1}{D_0}
		(t \ln t)^{-1}
		\left( 1+ O(R^{-\frac{1}{2} }) \right) \le \beta(t) \le
		-
		\frac{D_0}{D_1}  (t \ln t)^{-1}
		\left( 1+ O(R^{-\frac{1}{2} }) \right) & &  \mbox{ \ if \ } \gamma = 2 ,
		\\
		&
	 \beta(t) \sim -v_{5,\gamma} & &
		\mbox{ \ if \ } \gamma > 2 .
	\end{aligned}
	\right.
\end{equation}
We will solve the system \eqref{mu1-xi-sys} in the space
\begin{equation}
B_{\dot{\mu}_1}:=\left\{
f\in C[t_0,\infty) \ | \
\|f\|_{\dot{\mu}_1} \le 1
\right\},
\quad
B_{\dot{\xi}}=\left\{ \vec{f}=\left( f_1,\dots, f_5 \right) \in C[t_0,\infty) \ | \ \| \vec{f} \|_{\dot{\xi}} \le 1 \right\}
\end{equation}
with the norm
\begin{equation}\label{mu-xi-norm}
\|f\|_{\dot{\mu}_1}:=\sup_{t\ge t_0} \left( \mu_{0*}^{\frac{1}{2}} v_{5,\tilde{\gamma}} R^{-\frac{2}{3}}
 \right)^{-1}(t)  \left| f(t) \right|,
 \quad
 \| \vec{f} \|_{\dot{\xi}}
 :=
 \sup\limits_{t\ge t_0} \left(
 \mu_{0*}^{\frac{1}{2}}  v_{5,\tilde{\gamma}} R^{-\frac{7}{4}}
 \right)^{-1}(t) | \vec{f}(t) | ,
\end{equation}
where
$$
\mu_{0*}^{\frac{1}{2}} v_{5,\tilde{\gamma}}
=
\begin{cases}
	t^{1-\gamma}, & \gamma<2
	\\
	t^{-1} \ln t,
	& \gamma =2
	\\
	t^{-\frac{\tilde{\gamma}}{2}} ,
	&
	\gamma >2
\end{cases}.
$$
For any $(\dot{\mu}_1, \dot{\xi} ) \in B_{\dot{\mu}_1} \times B_{\dot{\xi}}$, it is easy to see that $\int_{\tilde{t}_0}^t \dot{\mu}_1(a) da$ and $\int_{\tilde{t}_0}^t \dot{\xi}(a) da$ in \eqref{mu1-xi-sys} are well-defined, and
\begin{equation}
\left|\mu_1  \right|
\lesssim
\begin{cases}
t^{2-\gamma} R^{-\frac{2}{3}},
&
\gamma<2
\\
\left( \ln t\right)^2 R^{-\frac{2}{3}}
,
&
\gamma=2
\\
t^{1-\frac{\tilde{\gamma}}{2}}
R^{-\frac{2}{3}}
,
&
\gamma>2
\end{cases}
\quad
\lesssim
\mu_{0*} R^{-\frac{2}{3}},
\quad
\left| \xi \right| \lesssim
\mu_{0*} R^{-\frac{7}{4}} .
\end{equation}
Thus, $\mu_1$, $\dot{\mu}_1$, $\xi$, $\dot{\xi}$ satisfy the assumption
\eqref{mu-xi-rou-bound} in Proposition \ref{outer-exist}.
By \eqref{nab-Psi0-upp}, \eqref{nabla-psi-est}, and $\gamma>\frac{3}{2}$, we get
\begin{equation}\label{vec-S-est}
	\left| \vec{\mathcal{S}}[\mu_1,\xi] \right|
	\lesssim
	\mu^{\frac{1}{2}}
	\left(
	\| \nabla_x \Psi_0(\cdot,t)\|_{L^\infty(\mathbb{R}^5)}
	+
	\| \nabla_x \psi(\cdot,t)\|_{L^\infty(\mathbb{R}^5) }
	\right)
	\left( |\mu|+ |\xi| \right)
	\lesssim
	\mu_{0*}^{\frac{1}{2}}
	v_{5,\tilde{\gamma}} R^{-2} \ln^2 R  .
\end{equation}
Using \eqref{psi-est}, \eqref{nab-Psi0-upp}, \eqref{Psi0-0-est} in order, we have
\begin{equation*}
\begin{aligned}
&
	\left|
	\mu^{\frac{1}{2}}
	\int_{B_{4R}}
	\psi(\mu y+\xi, t) U(y)^{\frac{4}{3}}
	Z_{6}(y) dy
	\right|
	\lesssim
	\mu_{0*}^{\frac{1}{2}} v_{5,\tilde{\gamma}} R^{-1} \ln^2 R
,
\\
&
\left| \mu^{\frac{1}{2}}
\int_{B_{4R}} \left(
\Psi_0(\mu y+\xi,t)
-
\Psi_0(0,t)
\right) U(y)^{\frac{4}{3}} Z_{6}(y)  dy \right|
\lesssim
\mu^{\frac{1}{2}}
\left( \mu+ |\xi|\right)
t^{-\frac{1}{2}} v_{5,\gamma}
\sim
\mu_{0*}^{\frac{3}{2}}
t^{-\frac{1}{2}}
v_{5,\gamma} ,
\\
&
\left| \left(
\mu^{\frac{1}{2}}
-
\mu_0^{\frac{1}{2}}
-
\frac{1}{2}
\mu_0^{-\frac{1}{2}} \mu_1
\right)
\Psi_0(0,t)
\int_{B_{4R}}	  U(y)^{\frac{4}{3}} Z_{6}(y)    dy \right|
\lesssim
\mu_{0*}^{\frac{1}{2}} v_{5,\gamma} R^{-\frac{4}{3}} ,
\end{aligned}
\end{equation*}
which implies
\begin{equation}\label{F-upp}
\left|\mathcal{F}[\mu_1,\xi](t) \right| \lesssim
\mu_{0*}^{\frac{1}{2}} v_{5,\tilde{\gamma}} R^{-\frac{3}{4}}
.
\end{equation}

Since $0<D_0\le D_1 <2D_0$, there exists $\epsilon_1>0$ sufficiently small so that $\frac{D_1}{D_0} (1+\epsilon_1) <2$. By taking $t_0$ sufficiently large, which can depend on $\gamma$, and using \eqref{beta-sim}, we obtain that
for $\gamma<2$,
\begin{equation*}
	\begin{aligned}
		&
		\left|
		\int_{\tilde{t}_0}^t \mathcal{F}[\mu_1,\xi](s)  e^{\int_t^s \beta(a)da }
		ds \right|
		\lesssim
		\int_{t_0}^t s^{1-\gamma}
		(\ln \ln s)^{-\frac{3}{4} } e^{\frac{D_1}{D_0} \left( 1-\frac{\gamma}{2}\right)\left(1+\epsilon_1\right) \int_s^t a^{-1} da }
		ds
		\\
		= \ &
		t^{\frac{D_1}{D_0} \left( 1-\frac{\gamma}{2}\right)\left(1+\epsilon_1\right)}
		\int_{t_0}^t s^{1-\gamma - \frac{D_1}{D_0} \left( 1-\frac{\gamma}{2}\right)\left(1+\epsilon_1\right)}
		(\ln \ln s)^{-\frac{3}{4}}
		ds
		\lesssim
		t^{2-\gamma} R^{-\frac{3}{4}}   ;
	\end{aligned}
\end{equation*}
for $\gamma=2$,
\begin{equation*}
\begin{aligned}
&
\left|
\int_{\tilde{t}_0}^t \mathcal{F}[\mu_1,\xi](s)  e^{\int_t^s \beta(a)da }
ds \right|
\lesssim
\int_{t_0}^t
s^{-1} \ln s
(\ln \ln s)^{-\frac{3}{4} }
 e^{\frac{D_1}{D_0} \left(1+\epsilon_1\right) \int_s^t \left( a\ln a\right)^{-1} da }
ds
\\
= \ &
(\ln t)^{\frac{D_1}{D_0} (1+\epsilon_1)}
\int_{t_0}^t
s^{-1} (\ln s)^{1-\frac{D_1}{D_0} (1+\epsilon_1)}
(\ln \ln s)^{-\frac{3}{4} }
ds
\\
= \ &
(\ln t)^{\frac{D_1}{D_0} (1+\epsilon_1)}
\int_{\ln t_0}^{\ln t}
z^{1-\frac{D_1}{D_0} (1+\epsilon_1)}
(\ln z)^{-\frac{3}{4} }
dz
\lesssim
(\ln t)^2 R^{-\frac{3}{4}}  ;
\end{aligned}
\end{equation*}
for $\gamma>2$,
\begin{equation*}
\left|
\int_{\tilde{t}_0}^t \mathcal{F}[\mu_1,\xi](s)  e^{\int_t^s \beta(a)da }
ds \right|
\lesssim
\int_t^\infty
s^{-\frac{\tilde{\gamma}}{2}}
\left( \ln \ln s\right)^{-\frac{3}{4} }
ds
\lesssim
t^{1-\frac{\tilde{\gamma}}{2}}
R^{-\frac{3}{4}} .
\end{equation*}
Thus, $\int_{\tilde{t}_0}^t \mathcal{F}[\mu_1,\xi](s)  e^{\int_t^s \beta(a)da }
ds$ is well-defined in \eqref{mu1-xi-sys}, and
\begin{equation}\label{F-int-upp}
	\left| \beta(t)
	\int_{\tilde{t}_0}^t \mathcal{F}[\mu_1,\xi](s)  e^{\int_t^s \beta(a)da }
	ds \right|
	\lesssim
	\mu_{0*}^{\frac{1}{2}} v_{5,\tilde{\gamma}} R^{-\frac{3}{4}}
	.
\end{equation}
Combining \eqref{vec-S-est}, \eqref{F-upp} and \eqref{F-int-upp}, we have
\begin{equation}\label{Si-est}
\left|	\mathcal{S}_{6}[\mu_1,\xi] \right| \lesssim  \mu_{0*}^{\frac{1}{2}} v_{5,\tilde{\gamma}} R^{-\frac{3}{4}},
	\quad
	\left| \vec{\mathcal{S} } [\mu_1,\xi]  \right|
\lesssim
\mu_{0*}^{\frac{1}{2}}
v_{5,\tilde{\gamma}} R^{-2} \ln^2 R
	,
\end{equation}
which implies
$\big( \mathcal{S}_{6} ,
\vec{\mathcal{S} }  \big)[\mu_1,\xi]
\in B_{\dot{\mu}_1} \times B_{\dot{\xi}} $.

For any sequence $(\dot{\mu}_{1}^{[j]},\dot{\xi}^{[j]})_{j\ge 1} \subset B_{\dot{\mu}_1} \times B_{\dot{\xi}}$, denote $\mu_1^{[j]} =\int_{\tilde{t}_0}^t \dot{\mu}_1^{[j]}(a) da $,
$
\xi^{[j]} =\int_{\tilde{t}_0}^t \dot{\xi}^{[j]}(a) da $. We set $\tilde{\dot{\mu}}_1^{[j]}:=\mathcal{S}_{6}[\mu_1^{[j]},\xi^{[j]}]$, $\tilde{\dot{\xi}}^{[j]}:= \vec{\mathcal{S} }[\mu_1^{[j]},\xi^{[j]}]$.
By the same method for deducing \eqref{Si-est}, we have
\begin{equation}\label{mu-xi-j-est}
|	\tilde{\dot{\mu}}_1^{[j]}  | \le C_1  \mu_{0*}^{\frac{1}{2}} v_{5,\tilde{\gamma}} R^{-\frac{3}{4}},
\quad
 |  \tilde{\dot{\xi}}^{[j]} |
\le
C_1 \mu_{0*}^{\frac{1}{2}}
v_{5,\tilde{\gamma}} R^{-2} \ln^2 R
\quad
\mbox{ \ for all \ }  j\ge 1
\end{equation}
with a constant $C_1>0$ independent of $j$.

For any compact subset $K \subset \subset [t_0,\infty)$, by the equation \eqref{mu1-xi-sys} and the space-time regularity for the outer solution $\psi$,
for all $j\ge 1$,
$\tilde{\dot{\mu}}_1^{[j]}$ and $\tilde{\dot{\xi}}^{[j]}$ are uniformly H\"older continuous in $K$. Since there exist countable compact sets to saturate $[t_0,\infty)$,
then up to a subsequence, for any compact set $K \subset \subset [t_0,\infty)$,
\begin{equation*}
\tilde{\dot{\mu}}_1^{[j]}
\rightarrow g, \quad \tilde{\dot{\xi}}^{[j]} \rightarrow
\vec{g} \mbox{ \ in \ } L^\infty(K)
\mbox{ \ as \ } j\rightarrow\infty
\end{equation*}
for some $g,~\vec{g} \in C[t_0,\infty)$. By \eqref{mu-xi-j-est}, we have
\begin{equation*}%\label{g-g-est}
	|g  | \le C_1  \mu_{0*}^{\frac{1}{2}} v_{5,\tilde{\gamma}} R^{-\frac{3}{4}},
	\quad
	| \vec{g}|
	\le
	C_1 \mu_{0*}^{\frac{1}{2}}
	v_{5,\tilde{\gamma}} R^{-2} \ln^2 R .
\end{equation*}
Thus, for any $\epsilon_1>0$, there exists $t_1$ sufficiently large such that for all $j\ge 1$,
\begin{equation*}
 \sup_{t\ge t_1} \left( \mu_{0*}^{\frac{1}{2}} v_{5,\tilde{\gamma}} R^{-\frac{2}{3}}
	\right)^{-1}(t)  \left| \left( \tilde{\dot{\mu}}_1^{[j]}  - g \right)(t) \right|
	+
	\sup\limits_{t\ge t_1} \left(
	\mu_{0*}^{\frac{1}{2}}  v_{5,\tilde{\gamma}} R^{-\frac{7}{4}}
	\right)^{-1}(t) \left| \left( \tilde{\dot{\xi}}^{[j]} -
	\vec{g} \right)(t) \right| <\epsilon_1 .
\end{equation*}
Additionally,
\begin{equation*}
	\lim\limits_{j\rightarrow \infty}
\bigg[ 	\sup_{t_0\le t \le t_1} \left( \mu_{0*}^{\frac{1}{2}} v_{5,\tilde{\gamma}} R^{-\frac{2}{3}}
	\right)^{-1}(t)  \left| \left( \tilde{\dot{\mu}}_1^{[j]}  - g \right)(t) \right|
	+
	\sup\limits_{t_0\le t \le t_1} \left(
	\mu_{0*}^{\frac{1}{2}}  v_{5,\tilde{\gamma}} R^{-\frac{7}{4}}
	\right)^{-1}(t) \left| \left( \tilde{\dot{\xi}}^{[j]} -
	\vec{g} \right)(t) \right| \bigg]  =0 .
\end{equation*}
Consequently, $\lim\limits_{j\rightarrow \infty} \left(\|\tilde{\dot{\mu}}_1^{[j]}  - g\|_{\dot{\mu}_1} + \| \tilde{\dot{\xi}}^{[j]} -
\vec{g} \|_{\dot{\xi}} \right) =0$, which implies
$\big( \mathcal{S}_{6} ,
\vec{\mathcal{S} }  \big)[\mu_1,\xi] $ is a compact mapping on $B_{\dot{\mu}_1} \times B_{\dot{\xi}}$.

By the Schauder fixed-point theorem, there exists a solution $(\dot{\mu}_1, \dot{\xi} ) \in B_{\dot{\mu}_1} \times B_{\dot{\xi}}$ for the system \eqref{mu1-xi-sys}.

\end{proof}

\medskip

\section{Solving the inner problem}

By \eqref{mu-xi-final}, \eqref{nab-Psi0-upp}, \eqref{psi-est}, and \eqref{til-v-def}, for $|y|\le 4R$,
\begin{equation*}
\begin{aligned}
&
\left|
\mu \dot{\xi} \cdot \left(\nabla U\right)(y)
	+\frac{7}{3}\mu^{\frac{3}{2}} U(y)^{\frac{4}{3}}
	\Big(
	\Psi_0(\mu y+\xi,t)
	-
	\Psi_0(0,t)
	+
	\psi(\mu y+\xi, t)
	\Big)
\right|
\\
\lesssim \ &
\mu_{0*}^{\frac{3}{2}} v_{5,\tilde{\gamma}} R^{-1} \ln^2 R
\langle y \rangle^{-4}
\sim    \tilde{v}_{5,\tilde{\gamma}}(\tau(t)) R^{-1}  \ln^2 R
\langle y \rangle^{-4}
 .
\end{aligned}
\end{equation*}
For brevity, denote $\tilde{H}[\phi]:=
\mathcal{H}\left[\psi\big[\phi,\mu_1[\phi],\xi[\phi] \big],\mu_0+\mu_1[\phi],\xi[\phi]\right]$. From \eqref{leading-H-est}, we have
\begin{equation}\label{H-final-upp}
| \tilde{H}[\phi] |
\lesssim
\tilde{v}_{5,\tilde{\gamma}}(\tau(t))
\langle y \rangle^{-3} .
\end{equation}

By Proposition \ref{mu1-xi-prop}, we can apply Proposition \ref{phi-upper-bdd} to the inner problem \eqref{inner-tau-eq}, and it suffices to solve the following fixed-point problem
\begin{equation*}
\phi =
\mathcal{T}_{\rm{in}}\big[ \tilde{H}[\phi] \big]
 .
\end{equation*}
Indeed, for any $\phi \in B_{\rm{in}}$, given $t_0$ (i.e. $\tau_0$) sufficiently large, by Proposition \ref{phi-upper-bdd}, we have
\begin{equation}\label{2023-0801-1}
	\langle y\rangle \left|\nabla_y \mathcal{T}_{\rm{in}}\big[ \tilde{H}[\phi] \big] \right|
	+
\left|  \mathcal{T}_{\rm{in}}\big[ \tilde{H}[\phi] \big] \right| \lesssim \tilde{v}_{5,\tilde{\gamma} }(\tau)R^{5}\ln R \langle y\rangle^{-6} ,
\quad
\left|  \mathcal{T}_{e_0}\big[ \tilde{H}[\phi] \big] \right|\lesssim \tilde{v}_{5,\tilde{\gamma} }(\tau_0)R(\tau_0) ,
\end{equation}
which implies $\mathcal{T}_{\rm{in}}\big[ \tilde{H}[\phi] \big] \in B_{\rm{in}}$ in particular.

For any sequence $(\phi_{j} )_{j\ge 1} \in B_{\rm{in}}$, denote $\tilde{\phi}_{j}:= \mathcal{T}_{\rm{in}}\big[ \tilde{H}[\phi_j] \big]$, $\tilde{e}_j:=   \mathcal{T}_{e_0}\big[ \tilde{H}[\phi_j] \big]$, which satisfies
\begin{equation*}
\begin{cases}
	\partial_{\tau} \tilde{\phi}_j=\Delta_y \tilde{\phi}_j +\frac{7}{3}U(y)^{\frac{4}{3}}
	\tilde{\phi}_j
	+  \tilde{H}[\phi_j]
	&
	\mbox{ \ in \ } \mathcal{D}_{4R} :=
	\left\{ (y,\tau) \ | \ \tau\in (\tau_0, \infty),
	\quad y\in B_{4R(t(\tau))}  \right\}
\\
\tilde{\phi}_j(\cdot,\tau_0) =
\tilde{e}_j Z_0
&
\mbox{ \ in \ } B_{4R(t_0)} .
\end{cases}
\end{equation*}
Repeating the process for deducing \eqref{H-final-upp} and \eqref{2023-0801-1}, one sees that there exists a constant $C_1$ independent of $j$ such that
\begin{equation}\label{2023-0801-2}
| \tilde{H}[\phi_j] |
\le C_1
\tilde{v}_{5,\tilde{\gamma}}(\tau)
\langle y \rangle^{-3},
\quad
	\langle y\rangle \left|\nabla_y \tilde{\phi}_j \right|
	+
	\left| \tilde{\phi}_j \right| \le C_1 \tilde{v}_{5,\tilde{\gamma} }(\tau)R^{5}\ln R \langle y\rangle^{-6} ,
	\quad
	\left| \tilde{e}_j \right|\le C_1 \tilde{v}_{5,\tilde{\gamma} }(\tau_0)R(t_0) .
\end{equation}
By the parabolic regularity theory, for any compact  set $K\subset \subset \mathcal{D}_{3R} \cup (B_{3R(t_0)} \times \{ \tau_0\} )$, it holds that $\| \phi_j\|_{C^{1+\ell, \frac{1+\ell}{2}}(K)} \le C_2$ with  a constant $C_2$ independent of $j$
and a constant  $\ell\in(0,1)$. By Arzel\`a-Ascoli theorem, up to a subsequence, there exists a function $g$ which is $C^1$ in space, such that
\begin{equation*}
\tilde{\phi}_j \rightarrow g,
\quad
\nabla_y \tilde{\phi}_j \rightarrow \nabla_y  g \mbox{ \ in \ }
L^\infty(K) \mbox{ \ as \ } j\rightarrow\infty.
\end{equation*}
By \eqref{2023-0801-2}, we have
\begin{equation*}%\label{2023-0801-2}
	\langle y\rangle \left|\nabla_y g \right|
	+
	\left| g \right| \le C_1 \tilde{v}_{5,\tilde{\gamma} }(\tau)R^{5}\ln R \langle y\rangle^{-6} \mbox{ \ in \ } \mathcal{D}_{3R} .
\end{equation*}
For any $\epsilon_1>0$, there exists $\tau_1$ sufficiently large such that
\begin{equation*}%\label{inner-norm}
\sup_{\tau>\tau_1, ~y\in B_{2R(t(\tau))}}
	\left(\tilde{v}_{5,\tilde{\gamma} }(\tau)R^{5}(t(\tau)) \ln^2 \left( R( t(\tau) ) \right) \right)^{-1}
	\langle y\rangle^{6}
	\big(\langle y\rangle \left|\nabla_y (\tilde{\phi}_j - g)(y,\tau) \right| + |(\tilde{\phi}_j - g)(y,\tau)| \big) <\epsilon_1 ,
\end{equation*}
and
\begin{equation*}%\label{inner-norm}
\lim\limits_{j \rightarrow \infty}	\sup_{\tau_0\le \tau\le \tau_1, ~y\in B_{2R(t(\tau))}}
	\left(\tilde{v}_{5,\tilde{\gamma} }(\tau)R^{5}(t(\tau)) \ln^2 \left( R( t(\tau) ) \right) \right)^{-1}
	\langle y\rangle^{6}
	\big(\langle y\rangle \left|\nabla_y (\tilde{\phi}_j - g)(y,\tau) \right| + |(\tilde{\phi}_j - g)(y,\tau)| \big) =0 .
\end{equation*}
Thus $\|\tilde{\phi}_j - g\|_{\rm{in}} \rightarrow 0$, which implies $\mathcal{T}_{\rm{in}}\big[ \tilde{H}[\phi] \big]$ is a compact mapping on $B_{\rm{in}}$. By the Schauder fixed-point theorem, there exists a solution $\phi\in B_{\rm{in}}$ and thus the construction is complete.

\medskip

\section{Properties of the solution $u$}

Recall $u$ given in \eqref{u-def}.  By \eqref{Psi0-rough-upp},
$\psi$ given by Proposition \ref{outer-exist}, $\phi$ solved in $B_{\rm{in}}$ (see \eqref{Bin-def}), $\mu_1,\xi$ given in Proposition \ref{mu1-xi-prop}, we have the validity of \eqref{u-behavior}. The initial value
\begin{equation*}
u(x,t_0) =
\mu^{-\frac{3}{2}}(t_0) U \left(\frac{x-\xi(t_0)}{\mu(t_0)} \right)
\eta\left( \frac{x-\xi(t_0)}{\sqrt{t_0}}  \right)
+
\Psi_0(x,t_0)
+
\mu^{-\frac{3}{2}}(t_0)
\phi\left(\frac{x-\xi(t_0)}{\mu(t_0)},t_0 \right)
\eta\left(\frac{x-\xi(t_0)}{\mu(t_0) R(t_0) }\right) .
\end{equation*}
Here $\Psi_0 > 0$. Denote $y(t_0)=\frac{x-\xi(t_0)}{\mu(t_0)}$, then
\begin{equation*}
U(y(t_0))-\phi(y(t_0),t_0)
\ge
15^{\frac{3}{4}} \langle y(t_0) \rangle^{-3}
-
C
\langle y(t_0) \rangle^{-6}
R^5(t_0) \ln^2 R(t_0)
\begin{cases}
t_0^{3-2\gamma},
& \frac{3}{2}<\gamma <2
\\
t_0^{-1} (\ln t_0)^3 ,
&\gamma =2
\\
t_0^{-\frac{\tilde{\gamma}}{2}},
&
\gamma>2
\end{cases}
>0
\end{equation*}
for $t_0$ large enough.
Therefore, $u(x,t_0)>0$, which implies $u>0$ by the maximum principle.
In addition, by Lemma \ref{initial-limit-lem}, we get \eqref{u-limit}. Finally, we conclude the proof of Theorem \ref{5d-main-th}.

 \medskip

\appendix

\section{Proof of Lemma \ref{Cauchy-est-x=0}}\label{Cauchy-est-App}

\begin{proof}[Proof of Lemma \ref{Cauchy-est-x=0}]
	For $\gamma<n$, $t\ge 1$,
	\begin{equation*}
		\begin{aligned}
			&
			\left(4\pi t\right)^{-\frac n2} \int_{\RR^n}
			e^{-\frac{|y|^2}{4t} }
			\langle y \rangle^{-\gamma} dy
			=
			(4\pi)^{-\frac n2}  t^{-\frac{\gamma}{2}} \int_{\RR^n}
			e^{-\frac{|z|^2}{4 } }
			\left( |z|^2 + t^{-1} \right)^{-\frac{\gamma}{2}} d z
			\\
			= \ &
			t^{-\frac{\gamma}{2}} (4\pi)^{-\frac n2} \int_{\RR^n}
			e^{-\frac{|z|^2}{4 } }
			|z|^{- \gamma } d z
			+ O\Big( t^{-\frac{\gamma}{2}}
			\begin{cases}
				t^{-1},   & \gamma <n-2
				\\
				t^{-1} \langle \ln t \rangle,  &
				\gamma =n-2
				\\
				t^{ \frac{\gamma -n }{2}},  &
				n-2 < \gamma <n
			\end{cases}
			\Big)
		\end{aligned}
	\end{equation*}
	since
\begin{equation*}
	\begin{aligned}
		&
		\left|
		\left(\int_{|z|\le t^{-\frac 12}} + \int_{|z| > t^{-\frac 12}}  \right)
		 \int_{\RR^n}
		e^{-\frac{|z|^2}{4 } }
		\left[ \left( |z|^2 + t^{-1} \right)^{-\frac{\gamma}{2}}
		-
		|z|^{- \gamma } \right]  d z
		\right|
		\\
		\lesssim \ &
		\int_{|z|\le t^{-\frac 12}}
		\begin{cases}
		t^{\frac{\gamma}{2}},
		& \gamma<0
		\\
		|z|^{-\gamma},
		& \gamma \ge 0
		\end{cases}  d z
		+
		\int_{|z| > t^{-\frac 12}}
		e^{-\frac{|z|^2}{4 } }
		|z|^{-\gamma-2}  t^{-1}  d z
		\\
		\lesssim \ &
		t^{ \frac{\gamma -n }{2}}
		+
		\begin{cases}
			t^{-1},   & \gamma <n-2
			\\
			t^{-1} \langle \ln t \rangle,  & \gamma =n-2
			\\
			t^{ \frac{\gamma -n }{2}} , & n-2 < \gamma <n
		\end{cases}
		\sim
		\begin{cases}
			t^{-1},   & \gamma <n-2
			\\
			t^{-1} \langle \ln t \rangle,  & \gamma =n-2
			\\
			t^{ \frac{\gamma -n }{2}},  & n-2 < \gamma <n .
		\end{cases}
	\end{aligned}
\end{equation*}
	
	For $\gamma=n$, $t\ge 1$,
	\begin{equation*}
		\left(4\pi t\right)^{-\frac n2} \int_{\RR^n}
		e^{-\frac{|y|^2}{4t} }
		\langle y \rangle^{-n } dy
		=
		t^{-\frac n2} \ln (1+t)
		\left(4\pi \right)^{-\frac n2} \frac{1}{2} \left|S^{n-1} \right|
		\left(1+ O\left( ( \ln (1+t)  )^{-1} \right) \right)
	\end{equation*}
	since
	\begin{equation*}
	\begin{aligned}
	&
		\left(4\pi t\right)^{-\frac n2} \int_{|y|\ge t^{\frac 12} }
		e^{-\frac{|y|^2}{4t} }
		\langle y \rangle^{-n } dy
		\sim
		t^{-\frac n2} \int_{\frac{1}{4} }^\infty
		e^{-z } z^{-1} d z,
\\
	&
	\left|	\left(4\pi t\right)^{-\frac n2} \int_{ |y| < t^{\frac 12}  }
		\left(e^{-\frac{|y|^2}{4t} } - 1 \right)
		\langle y \rangle^{-n } dy
		\right|
		\lesssim
		t^{-\frac n2} \int_{ |y| < t^{\frac 12}  }
		t^{-1} |y|^2
		\langle y \rangle^{-n } dy
		\sim
		t^{-\frac n2} ,
		\\
	&	
\begin{aligned}
	\left(4\pi t\right)^{-\frac n2} \int_{ |y| < t^{\frac 12}  }
	\langle y \rangle^{-n } dy
	= \ &
	t^{-\frac n2}
	\left(4\pi \right)^{-\frac n2} \frac{1}{2} \left|S^{n-1}\right| \int_0^{t }
	\left[
	\frac{z^{\frac{n-2}{2}}}{(1+z)^{\frac{n}{2}}}
	-
	\frac{1}{1+z}
	+
	\frac{1}{1+z}
	\right]
	d z
	\\
	= \ &
	t^{-\frac n2}
	\left(4\pi \right)^{-\frac n2} \frac{1}{2} \left|S^{n-1}\right|
	\left(\ln (1+t) + O(1) \right) .
\end{aligned}
	\end{aligned}
	\end{equation*}
	
	For $\gamma>n$, $t\ge 1$,
	\begin{equation*}
		\left(4\pi t\right)^{-\frac n2} \int_{\RR^n}
		e^{-\frac{|y|^2}{4t} }
		\langle y \rangle^{-\gamma} dy
		=
		t^{-\frac n2}
		\left(4\pi \right)^{-\frac n2} \int_{\RR^n}
		\langle y \rangle^{-\gamma} dy
		+
		O
		\Big(
		t^{-\frac n2}
		\begin{cases}
			t^{\frac{n-\gamma}{2}},
			&
			n<\gamma<n+2
			\\
			t^{-1}  \langle \ln t \rangle,
			&
			\gamma=n+2
			\\
			t^{-1} ,
			&
			\gamma>n+2
		\end{cases}
		\Big)
	\end{equation*}
	since
	\begin{equation*}
		\begin{aligned}
			&
			\left|\left(4\pi t\right)^{-\frac n2} \left(\int_{|y|\le t^{\frac 12}} + \int_{|y|> t^{\frac 12}}  \right)
			\left(
			e^{-\frac{|y|^2}{4t} } -1\right)
			\langle y \rangle^{-\gamma} dy
			\right|
			\\
			\lesssim \ &
			t^{-\frac n2}
			\left(  \int_{|y|\le t^{\frac 12}}
			t^{-1} |y|^2
			\langle y \rangle^{-\gamma} dy
			+ \int_{|y|> t^{\frac 12}}
			\langle y \rangle^{-\gamma} dy
			\right)
			\lesssim
			t^{-\frac n2}
			\begin{cases}
				t^{\frac{n-\gamma}{2}},
				&
				n<\gamma<n+2
				\\
				t^{-1}  \langle \ln t \rangle,
				&
				\gamma=n+2
				\\
				t^{-1} ,
				&
				\gamma>n+2 .
			\end{cases}
		\end{aligned}
	\end{equation*}
	
\end{proof}

\medskip

\section{The limitation of $\int_{\mathbb{R}^n}
	e^{-A|x-y|^2}
	\langle y \rangle^{-b} dy$}

\begin{lemma}\label{initial-limit-lem}
	
	For $n>0$, $A>0$, $b\in \mathbb{R}$, we have
	\begin{equation*}
		\begin{aligned}
			&
			\int_{\mathbb{R}^n}
			e^{-A|x-y|^2}
			\langle y \rangle^{-b} dy
			=
			\langle x \rangle^{-b}
			\int_{\mathbb{R}^n} e^{-A|z|^2}
			dz
			\\
			&
			\qquad
			+
			O\bigg(
			|x|
			\langle x \rangle^{-b-2}
			\left(
			|x|^{n+1}\1_{|x|\le 1}
			+
			\1_{|x|> 1}
			\right)
			+
			\int_{|z|> \frac{|x|}{2}} e^{-A|z|^2}
			\left(|z|+\langle x \rangle\right)^{\max\left\{0,-b \right\} }
			dz
			\bigg) .
		\end{aligned}
	\end{equation*}
	
\end{lemma}

\begin{proof}
	
	\begin{equation*}
		\int_{\mathbb{R}^n}
		e^{-A|x-y|^2}
		\langle y \rangle^{-b} dy
		=
		\int_{\mathbb{R}^n} e^{-A|z|^2}
		\left[
		\left( 1+|x|^2\right)^{-\frac{b}{2}}
		+
		\left( 1+|x-z|^2\right)^{-\frac{b}{2}}
		-
		\left( 1+|x|^2\right)^{-\frac{b}{2}}
		\right]
		dz .
	\end{equation*}
	Here, we estimate
	\begin{equation*}
		\begin{aligned}
			&
			\left|
			\int_{|z|\le \frac{|x|}{2}} e^{-A|z|^2}
			\left[
			\left( 1+|x-z|^2\right)^{-\frac{b}{2}}
			-
			\left( 1+|x|^2\right)^{-\frac{b}{2}}
			\right]
			dz
			\right|
			\\
			= \ &
			\left|
			-\frac{b}{2}
			\int_{|z|\le \frac{|x|}{2}} e^{-A|z|^2}
			\left[
			1+\theta |x-z|^2
			+
			(1-\theta) |x|^2
			\right]^{-\frac{b}{2} -1}
			\left(
			|x-z| - |x|
			\right)
			\left(
			|x-z| + |x|
			\right)
			dz
			\right|
			\\
			\lesssim \ &
			|x|
			\langle x \rangle^{-b-2}
			\int_{|z|\le \frac{|x|}{2}} e^{-A|z|^2}
			|z|
			dz
			\sim
			|x|
			\langle x \rangle^{-b-2}
			\left(
			|x|^{n+1}\1_{|x|\le 1}
			+
			\1_{|x|> 1}
			\right)
		\end{aligned}
	\end{equation*}
	with a parameter $\theta\in [0,1]$;
	\begin{equation*}
		\begin{aligned}
			&
			\left|
			\int_{|z|> \frac{|x|}{2}} e^{-A|z|^2}
			\left[
			\left( 1+|x-z|^2\right)^{-\frac{b}{2}}
			-
			\left( 1+|x|^2\right)^{-\frac{b}{2}}
			\right]
			dz
			\right|
			\\
			\lesssim \ &
			\begin{cases}
				\int_{|z|> \frac{|x|}{2}} e^{-A|z|^2}  dz ,
				& b\ge 0
				\\
				\int_{|z|> \frac{|x|}{2}} e^{-A|z|^2}
				\left(
				|z|^{-b}
				+
				\langle x \rangle^{-b}
				\right)
				dz ,
				&
				b<0
			\end{cases}
			\sim
			\int_{|z|> \frac{|x|}{2}} e^{-A|z|^2}
			\left(|z|+\langle x \rangle\right)^{\max\left\{0,-b \right\} }
			dz  .
		\end{aligned}
	\end{equation*}
	
\end{proof}

\bigskip

\section*{Acknowledgements}

Z. Li is funded by Natural Science Foundation of Hebei Province, No. A2022205007 and by Science and Technology Project of Hebei Education Department, No. QN2022047. J. Wei is partially supported by NSERC of Canada.


\bigskip


\begin{thebibliography}{10}

\bibitem{ageno2023infinite}
Giacomo Ageno and Manuel del Pino.
\newblock Infinite time blow-up for the three dimensional energy critical heat
  equation in bounded domain.
\newblock {\em arXiv preprint arXiv:2301.10442}, 2023.

\bibitem{Chen-Li91}
Wen~Xiong Chen and Congming Li.
\newblock Classification of solutions of some nonlinear elliptic equations.
\newblock {\em Duke Math. J.}, 63(3):615--622, 1991.

\bibitem{Green16JEMS}
Carmen Cort\'{a}zar, Manuel del Pino, and Monica Musso.
\newblock Green's function and infinite-time bubbling in the critical nonlinear
  heat equation.
\newblock {\em J. Eur. Math. Soc. (JEMS)}, 22(1):283--344, 2020.

\bibitem{18Euler}
Juan Davila, Manuel Del~Pino,Monica Musso, and Juncheng Wei.
\newblock Gluing {M}ethods for {V}ortex {D}ynamics in {E}uler {F}lows.
\newblock {\em Arch. Ration. Mech. Anal.}, 235(3):1467--1530, 2020.

\bibitem{Eulerfrog}
Juan Davila, Manuel del Pino, Monica Musso, and Juncheng Wei.
\newblock Leapfrogging vortex rings for the 3-dimensional incompressible
  {E}uler equations.
\newblock {\em arXiv preprint arXiv:2207.03263}, 2022.

\bibitem{17HMF}
Juan D\'{a}vila, Manuel del Pino, and Juncheng Wei.
\newblock Singularity formation for the two-dimensional harmonic map flow into
  {$S^2$}.
\newblock {\em Invent. Math.}, 219(2):345--466, 2020.

\bibitem{173D}
Manuel del Pino, Monica Musso, and Juncheng Wei.
\newblock Infinite-time blow-up for the 3-dimensional energy-critical heat
  equation.
\newblock {\em Anal. PDE}, 13(1):215--274, 2020.

\bibitem{del2018sign}
Manuel del Pino, Monica Musso, Juncheng Wei, and Youquan Zheng.
\newblock Sign-changing blowing-up solutions for the critical nonlinear heat
  equation.
\newblock {\em Ann. Sc. Norm. Super. Pisa Cl. Sci. (5)}, 21:569--641, 2020.

\bibitem{FilaKing12}
Marek Fila and John~R. King.
\newblock Grow up and slow decay in the critical {S}obolev case.
\newblock {\em Netw. Heterog. Media}, 7(4):661--671, 2012.

\bibitem{Fila08JMAA}
Marek Fila, John~R. King, Michael Winkler, and Eiji Yanagida.
\newblock Linear behaviour of solutions of a superlinear heat equation.
\newblock {\em J. Math. Anal. Appl.}, 340(1):401--409, 2008.

\bibitem{Fila08MA}
Marek Fila, Michael Winkler, and Eiji Yanagida.
\newblock Slow convergence to zero for a parabolic equation with a
  supercritical nonlinearity.
\newblock {\em Math. Ann.}, 340(3):477--496, 2008.

\bibitem{Fujita66}
Hiroshi Fujita.
\newblock On the blowing up of solutions of the {C}auchy problem for
  {$u_{t}=\Delta u+u^{1+\alpha }$}.
\newblock {\em J. Fac. Sci. Univ. Tokyo Sect. I}, 13:109--124 (1966), 1966.

\bibitem{King03JDE}
Victor~A. Galaktionov and John~R. King.
\newblock Composite structure of global unbounded solutions of nonlinear heat
  equations with critical {S}obolev exponents.
\newblock {\em J. Differential Equations}, 189(1):199--233, 2003.

\bibitem{Vazquez97CPAM}
Victor~A. Galaktionov and Juan~L. Vazquez.
\newblock Continuation of blowup solutions of nonlinear heat equations in
  several space dimensions.
\newblock {\em Comm. Pure Appl. Math.}, 50(1):1--67, 1997.

\bibitem{Gidas-Spruck81}
B.~Gidas and J.~Spruck.
\newblock Global and local behavior of positive solutions of nonlinear elliptic
  equations.
\newblock {\em Comm. Pure Appl. Math.}, 34(4):525--598, 1981.

\bibitem{Gui01JDE}
Changfeng Gui, Wei-Ming Ni, and Xuefeng Wang.
\newblock Further study on a nonlinear heat equation.
\newblock volume 169, pages 588--613. 2001.
\newblock Special issue in celebration of Jack K. Hale's 70th birthday, Part 4
  (Atlanta, GA/Lisbon, 1998).

\bibitem{Kavian87AIHP}
Otared Kavian.
\newblock Remarks on the large time behaviour of a nonlinear diffusion
  equation.
\newblock {\em Ann. Inst. H. Poincar\'{e} Anal. Non Lin\'{e}aire},
  4(5):423--452, 1987.

\bibitem{Kawanago96}
Tadashi Kawanago.
\newblock Asymptotic behavior of solutions of a semilinear heat equation with
  subcritical nonlinearity.
\newblock {\em Ann. Inst. H. Poincar\'{e} C Anal. Non Lin\'{e}aire},
  13(1):1--15, 1996.

\bibitem{LeeNi}
Tzong-Yow Lee and Wei-Ming Ni.
\newblock Global existence, large time behavior and life span of solutions of a
  semilinear parabolic cauchy problem.
\newblock {\em Transactions of the American Mathematical Society},
  333(1):365--378, 1992.

\bibitem{Ni84JDE}
Wei-Ming Ni, Paul~E. Sacks, and John Tavantzis.
\newblock On the asymptotic behavior of solutions of certain quasilinear
  parabolic equations.
\newblock {\em J. Differential Equations}, 54(1):97--120, 1984.

\bibitem{Polacik2011ARMA}
P.~Pol\'{a}\v{c}ik.
\newblock Threshold solutions and sharp transitions for nonautonomous parabolic
  equations on {$\Bbb R^N$}.
\newblock {\em Arch. Ration. Mech. Anal.}, 199(1):69--97, 2011.

\bibitem{Polavcik-Quittner06}
Peter Pol\'{a}\v{c}ik and Pavol Quittner.
\newblock A {L}iouville-type theorem and the decay of radial solutions of a
  semilinear heat equation.
\newblock {\em Nonlinear Anal.}, 64(8):1679--1689, 2006.

\bibitem{Polacik07Indiana}
Peter Pol\'{a}\v{c}ik, Pavol Quittner, and Philippe Souplet.
\newblock Singularity and decay estimates in superlinear problems via
  {L}iouville-type theorems. {II}. {P}arabolic equations.
\newblock {\em Indiana Univ. Math. J.}, 56(2):879--908, 2007.

\bibitem{Polacik03MA}
Peter Pol\'{a}\v{c}ik and Eiji Yanagida.
\newblock On bounded and unbounded global solutions of a supercritical
  semilinear heat equation.
\newblock {\em Math. Ann.}, 327(4):745--771, 2003.

\bibitem{Quittner08}
Pavol Quittner.
\newblock The decay of global solutions of a semilinear heat equation.
\newblock {\em Discrete Contin. Dyn. Syst.}, 21(1):307--318, 2008.

\bibitem{Quittner2017}
Pavol Quittner.
\newblock Threshold and strong threshold solutions of a semilinear parabolic
  equation.
\newblock {\em Adv. Differential Equations}, 22(7-8):433--456, 2017.

\bibitem{Quittner21}
Pavol Quittner.
\newblock Optimal {L}iouville theorems for superlinear parabolic problems.
\newblock {\em Duke Math. J.}, 170(6):1113--1136, 2021.

\bibitem{Souplet07book}
Pavol Quittner and Philippe Souplet.
\newblock {\em Superlinear parabolic problems}.
\newblock Birkh\"{a}user Advanced Texts: Basler Lehrb\"{u}cher. [Birkh\"{a}user
  Advanced Texts: Basel Textbooks]. Birkh\"{a}user/Springer, Cham, 2019.
\newblock Blow-up, global existence and steady states, Second edition of [
  MR2346798].

\bibitem{17halfHMF}
Yannick Sire, Juncheng Wei, and Youquan Zheng.
\newblock Infinite time blow-up for half-harmonic map flow from {$\Bbb R$} into
  {$\Bbb S^1$}.
\newblock {\em Amer. J. Math.}, 143(4):1261--1335, 2021.

\bibitem{Suzuki1999}
Ryuichi Suzuki.
\newblock Asymptotic behavior of solutions of quasilinear parabolic equations
  with slowly decaying initial data.
\newblock {\em Adv. Math. Sci. Appl.}, 9(1):291--317, 1999.

\bibitem{Suzuki08Indiana}
Takashi Suzuki.
\newblock Semilinear parabolic equation on bounded domain with critical
  {S}obolev exponent.
\newblock {\em Indiana Univ. Math. J.}, 57(7):3365--3396, 2008.

\bibitem{infi4d}
Juncheng Wei, Qidi Zhang, and Yifu Zhou.
\newblock On {F}ila-{K}ing {C}onjecture in dimension four.
\newblock {\em arXiv preprint arXiv:2210.04352}, 2022.

\bibitem{tri}
Juncheng Wei, Qidi Zhang, and Yifu Zhou.
\newblock Trichotomy dynamics of the 1-equivariant harmonic map flow.
\newblock {\em arXiv preprint arXiv:2301.09221}, 2023.

\end{thebibliography}
\end{document}